\documentclass{article}
\usepackage[a4paper, lmargin=0.1\paperwidth, rmargin=0.1\paperwidth, tmargin=0.1111\paperheight, bmargin=0.1111\paperheight]{geometry} 
\usepackage{graphicx} % Required for inserting images
\usepackage{amsmath}
\usepackage{caption}
\usepackage{float}
\usepackage{amsfonts}
\usepackage{subcaption}
\usepackage{chngcntr}
\counterwithin{figure}{section}
\usepackage{hyperref}
\usepackage[toc,page]{appendix}
\usepackage{xcolor}
\usepackage{amsthm}
\usepackage{comment}

\newtheorem{thm}{Theorem}[section]

\newtheorem{lem}[thm]{Lemma}
\newtheorem{cor}[thm]{Corollary}
\newtheorem{remark}[thm]{Remark}

\numberwithin{equation}{section}

\newcommand{\slp}{\mathcal{S}}
\newcommand{\dlp}{\mathcal{D}}
\newcommand{\vp}{\mathcal{V}}
\newcommand{\norm}[1]{\left \lVert #1 \right \rVert}
\newcommand{\bfx}{\mathbf{x}}
\newcommand{\sn}{\sigma_0}
\newcommand{\kb}{\kappa_0}
\newcommand{\yt}{\gamma^{(3)}_0}
\newcommand{\yf}{\gamma^{(4)}_0}
\newcommand{\dhat}{\delta_*}
\newcommand{\R}{\mathbb{R}}

\DeclareMathOperator\erf{erf}
\DeclareMathOperator\erfc{erfc}

\begin{document}

\title{Fast unified evaluation of layer and volume potentials for the 2D modified Helmholtz equation}
\date{}

\author{
Edith Frisk G{\"a}rtner\textsuperscript{1,2}, 
Fredrik Fryklund\textsuperscript{1},
Anna-Karin Tornberg\textsuperscript{1}
}

\maketitle

\footnotetext[1]{Department of Mathematics, KTH Royal Institute of Technology, Stockholm, Sweden.}
\footnotetext[2]{Email: \texttt{egartner@kth.se}.}

\begin{abstract}
\noindent We present a fast and accurate potential theory-based method for the two-dimensional modified Helmholtz equation, treating the involved singular and nearly singular layer evaluations together with volume potentials within a single computational framework. The method is based on a decomposition of the free-space Green's function into a short-range local part and a smooth long-range part. The long-range contribution is evaluated efficiently using the non-uniform fast Fourier transform (NUFFT), while the local contribution is treated by asymptotic expansions. For the layer potentials, an intermediate telescoping sum over dyadic refinement levels is added, where the resulting difference kernels are smooth and rapidly decaying, allowing the dyadic levels to be evaluated without specialized quadrature rules. The volume potential is evaluated on triangular cut-cell meshes, where the mesh only enters the scheme as quadrature rule for smooth data. This makes the method robust with respect to small and distorted mesh cells, without the need for stabilization or cell-merging techniques. Numerical experiments demonstrate the expected convergence rates, high throughput of the potential evaluations, and robustness with respect to mesh quality. 
\end{abstract}

\section{Introduction}
In this paper, we present a unified kernel-splitting framework for the fast and accurate evaluation of layer and volume potentials for the modified Helmholtz equation in complex geometries. The method uses the same evaluation strategy for targets on the boundary, close to the boundary and away from the boundary, avoiding the need for regime-dependent quadrature rules. 
Since the volume potential arises from the source term in the PDE, the volume discretization enters only through quadrature of that term. This makes the method insensitive to mesh-cell shapes and compatible with cut-cell meshes, rendering it particularly attractive for complex or moving geometries. 
Combined, these features render the framework largely agnostic to the underlying discretization, requiring only quadrature rules that accurately integrate smooth functions over the boundary and the domain. 
%Combined, these features make the method easy to apply, requiring only quadrature rules for smooth integration over the boundary and the domain, and making it particularly convenient for volume potentials, where the method is largely agnostic to the underlying discretization.

%Importantly, the volume discretization only enters through quadrature of the source term, making the approach insensitive to mesh-cell shapes and compatible with cut-cell meshes, which is attractive for complex or moving geometries. Combined, these features make the method easy to apply, in particular for computing volume potentials, since it is agnostic to which smooth quadrature rule is provided.

One important motivation to consider the modified Helmholtz equation comes from elliptic marching, or Rothe's method, for time-dependent parabolic problems  \cite{chapko1997rothes, fryklund2020integral, kropinski2011rothe}. Implicit time discretizations of such equations often lead to elliptic boundary value problems that must be solved at each time step. For example, implicit time-stepping of the heat equation gives an equation of the form
\begin{align}\label{mod_helm}
    (-\Delta + \alpha^2)u(\bfx) = f(\bfx), \quad \bfx\in\Omega,
\end{align}
 which is the modified Helmholtz equation, also known as the Yukawa or the screened Poisson equation. Here $\alpha^2$ is proportional to the inverse of the time step. Since such problems must be solved repeatedly, and often in complex or moving geometries, fast and accurate solvers are needed. 

 Integral equation methods are attractive in this setting because they avoid discretizing the differential operator on the mesh. Instead, the solution of an inhomogenous boundary value problem is represented as the sum of a particular solution, in the form of a volume potential, and a homogeneous solution, represented by layer potentials. In two dimensions, the free-space Green's function for \eqref{mod_helm} is 
\begin{align}\label{Greens}
    G(\bfx) = \frac{1}{2\pi}K_0(\alpha\norm{\bfx}),
\end{align}
where $\|\bfx\|$ is the Euclidean norm, $K_n(\cdot)$ are the modified Bessel functions of the second kind and $n$th order, and the corresponding volume potential is
\begin{align}\label{vp}
    \vp[f](\bfx) = \int_\Omega G(\bfx-\bfx')\,f(\bfx')\,\mathrm{d}\bfx',
    \quad \bfx\in\Omega.
\end{align}
For a given source term $f$, the volume potential is evaluated, and the unknowns are layer densities defined only on the boundary $\partial\Omega$, that will be determined evoking the boundary conditions.  Thus the dimension of the problem is reduced by one. A computational challenge is that the Green's function is singular at the origin. Special care is therefore needed when evaluating layer potentials at points on or close to the boundary. For volume potentials, the singularity will always be inside, or on the boundary of, the integration domain. Moreover, direct evaluation of the potentials at $N$ target points from $N$ sources has computational cost $\mathcal{O}(N^2)$, motivating the need for fast algorithms.

 Several approaches have been developed to address the challenges that arise for the modified Helmholtz equation, and for elliptic PDEs more generally, when $f$ is available only as volumetric data on an unstructured mesh. In particular in the context of fast algorithms, such as in \cite{cheng2006adaptive} where an adaptive FMM-accelerated integral equation method for the inhomogeneous modified Helmholtz equation in two dimensions is developed. This FMM is further leveraged in the context of elliptic marching \cite{kropinski2011rothe} with the discretization of the forced isotropic heat equation in time. It results in solving a sequence of inhomogeneous modified Helmholtz problems represented using a volume potential and a double layer potential, with computations accelerated by the FMM. In that formulation, the volume potential is evaluated over a simple computational box, which requires the right-hand side $f$ to be available as a compactly supported function on that box, or for such an extension to be supplied. In \cite{fryklund2020integral} this idea is combined with a high-regularity extension of the right-hand side from the physical domain $\Omega$ to a compactly supported function $f^e$ on a simple enclosing box. This permits the particular solution to be evaluated as a free-space volume potential on the box, while the boundary conditions are imposed through a homogeneous boundary integral equation. 

Among other approaches are dual reciprocity methods, where the source term is approximated by basis functions for which particular solutions are available analytically \cite{chen1998dual,muleshkov2005particular}. A transform-based approach was recently developed in \cite{croix2026lightning}, which uses the Laplace transform to reduce the heat equation to modified Helmholtz problems and solve the resulting elliptic equations using the Lightning method \cite{gopal2019new}. In addition, Yukawa kernels have been treated in the fast multipole literature \cite{greengard2002new}. Other related numerical approaches include kernel-free boundary integral methods \cite{xie2018fourth} and method of fundamental solutions approaches \cite{chen2015fast}. Finally, some fast Poisson solvers may also be adaptable to the modified Helmholtz equation, such as the fully adaptive high-order solver of \cite{fortunato2025fully}, which combines a fast box solver in the bulk with spectral collocation in a boundary-conforming strip.

In the above, we mainly discussed methods to handle the volume potential. 
General high-order approaches for singular and nearly singular layer potential evaluation include kernel-split product integration methods \cite{helsing2008evaluation}, and quadrature by expansion (QBX) methods \cite{klockner2013quadrature}. 
For the modified Helmholtz equation, the parameter dependence introduces an additional difficulty: as $\alpha$ increases, the layer kernels become more localized. The kernel split quadrature is based on writing the kernel on a form with smooth functions multiplying explicit singularities, and large errors arise if these smooth functions cannot be locally well approximated by polynomials. 
This issue was studied in \cite{fryklund2022adaptive}, where an adaptive refinement strategy for kernel-split product integration was introduced for parameter-dependent layer potentials. 

The method developed in this paper extends the fast and geometrically flexible algorithm for the Poisson equation in \cite{fryklund2026lightweight} to the modified Helmholtz equation. The method is based on an integral representation of the Green's function, which allows the kernel to be decomposed into a short-range and a long-range contribution. The long-range component can be evaluated efficiently in Fourier space using the non-uniform fast Fourier transform (NUFFT) \cite{dutt1993fast, greengard2004accelerating}. The near-singular short-time contribution is further subdivided into two parts: one that is treated analytically using local asymptotic expansions, and one represented as a telescoping sum over dyadic refinement levels. This decomposition enables high accuracy in the local expansions while controlling the computational cost of the NUFFT. This combination of kernel splitting and asymptotic analysis is closely related to earlier work on heat potentials \cite{greengard1990fast, li2009high, wang2019hybrid} and the dual-space multilevel kernel-splitting (DMK) framework \cite{jiang2025dual}.

An important feature of the method is its geometrical flexibility. In this paper, we make use of this feature by incorporating triangular cut-cell meshes. In this approach, the domain is embedded in a uniform background grid which is intersected with the boundary to produce a boundary fitted triangulation. Such meshes have an advantage in problems involving time-dependent geometries, since geometric changes only affect the boundary intersecting elements, and so remeshing becomes less costly. Moreover, because the particular solution is computed through a volume potential formulation, the method is largely insensitive to degenerate and poorly shaped elements, avoiding the stability issues that often arise in finite element discretizations.

The main contribution of this work is the extension of the heat kernel based kernel-splitting framework of \cite{fryklund2026lightweight} to the modified Helmholtz equation. In particular, we derive asymptotic expansions for the short-time contributions of the single layer, double layer, and volume potentials associated with the modified Helmholtz kernel. Compared to the Poisson case, the presence of the parameter $\alpha$ modifies both the local asymptotics and the Fourier space history term, and we analyze how this parameter affects the accuracy and efficiency of the method. We also combine the resulting method with triangular cut-cell meshes, demonstrating applicability to geometrically complex domains.

The remainder of the paper is organized as follows. In Section \ref{sec:ker_split}, we introduce the integral representation of the kernel and the decomposition of volume and layer potentials into a local and a history part. In Section \ref{sec:disc} we describe the triangular cut-cell meshes of the domain and the boundary quadrature, as well as the discretization of the integral equation. In Section \ref{sec:local} we derive the asymptotic expansions for all potentials, as well as the dyadic refinement strategy used for the layer potentials. Section \ref{sec:hist} describes the Fourier evaluation of the long-range contributions using the NUFFT. Finally, Section \ref{sec:num} presents numerical results, including validation of asymptotic expansions and robustness experiments with respect to mesh quality. 

%%%%%%%%%%%%%%%%%%%%%%%%%%%%%%
%%%%%%%%%%%%%%%%%%%%%%%%%%%%%%
%%%%%%%%%%%%%%%%%%%%%%%%%%%%%%
\section{Potential formulation and kernel splitting}\label{sec:ker_split}

Consider equation \eqref{mod_helm}, with $\Omega\subset\mathbb{R}^2$ either an interior or an exterior domain.  We consider both Dirichlet boundary conditions
\begin{align}
    u(\bfx) &= g(\bfx), \quad \bfx\in\partial\Omega, \label{Dirichlet}
\end{align}
and Neumann boundary conditions
\begin{align}
    \frac{\partial u}{\partial \boldsymbol{\nu}}(\bfx) &= g(\bfx), \quad \bfx\in\partial\Omega, \label{Neumann}
\end{align}
where $\boldsymbol{\nu}$ denotes the outward unit normal on the boundary $\partial\Omega$.

Using the free-space Green's function \eqref{Greens}, we define the volume potential \eqref{vp}, the single layer potential,
\begin{align}\label{slp}
    \slp[\sigma](\bfx) & = \int_{\partial\Omega} G(\bfx-\bfx')\,\sigma(\bfx')\,\mathrm{d}s_{\bfx'}, \quad \bfx\in\Omega,
\end{align}
 and the double layer potential,
\begin{align}\label{dlp}
    \dlp[\mu](\bfx) &= \int_{\partial\Omega} 
    \frac{\partial G}{\partial \nu_{\bfx'}}\!\left(\bfx-\bfx'\right)
    \mu(\bfx')\,\mathrm{d}s_{\bfx'}, \quad \bfx\in\Omega.
\end{align}

For the double layer potential, the limiting values as $\bfx$ approaches the boundary satisfy the classical jump relations 
\begin{align}\label{jumprel}
    \lim_{\mathbf{y}\to\bfx\in\partial\Omega, \mathbf{y}\in\Omega^{\pm}} 
    \dlp[\mu](\mathbf{y}) 
    = \dlp[\mu](\bfx) \pm \frac{1}{2}\mu(\bfx),
\end{align}
where $\Omega^+$ and $\Omega^-$ denote the exterior and interior of $\Omega$, respectively. In addition, the double layer kernel has a bounded diagonal limit,
\begin{align}\label{dlp_diag_limit}
\lim_{\bfx'\to\bfx}\frac{\partial G}{\partial \nu_{\bfx'}}\left(\bfx-\bfx'\right)
=-\frac{\kappa(\bfx)}{4\pi},\qquad \bfx\in\partial\Omega,
\end{align}
where $\kappa$ is the signed curvature of $\partial\Omega$. 

From potential theory \cite{kress2014linear}, we know that the solution to the Dirichlet problem can be represented by
\begin{align}\label{sol_dirichlet}
    u(\bfx) = \vp[f](\bfx) + \dlp[\mu](\bfx), \quad \bfx\in\Omega,
\end{align}
where $\mu$ is an unknown layer density. Taking the limit of \eqref{sol_dirichlet} as $\bfx$ approaches the boundary and applying the jump relations~\eqref{jumprel} yields the boundary integral equation
\begin{align}\label{bie_dir}
    -\frac{1}{2}\mu(\bfx') + \dlp[\mu](\bfx') = g(\bfx') - \vp[f](\bfx'),
    \quad \bfx'\in\partial\Omega,
\end{align}
which is solved for $\mu$. For the Neumann problem, Green’s representation formula gives
\begin{align}
\label{sol_neumann}
    u(\bfx) = \vp[f](\bfx) + \slp\left[\frac{\partial u}{\partial \nu}\right](\bfx) - \dlp[u](\bfx), \quad \bfx\in\Omega,
\end{align}
and the corresponding boundary integral equation
\begin{align}\label{bie_neu}
    \frac{1}{2}u(\bfx') + \dlp[u](\bfx') = \vp[f](\bfx') + \slp[g](\bfx'),
    \quad \bfx'\in\partial\Omega.
\end{align}

\subsection{Decomposition of volume and layer potentials}\label{sec:decomp}
The free-space Green's function \eqref{Greens} can equivalently be expressed through an integral representation \cite{jiang2025dual} as
\begin{align}\label{ker_greens}
        \frac{1}{2\pi}K_0(\alpha \norm{\mathbf{x}}) = \int_0^\infty \frac{e^{-\frac{\norm{\bfx}^2}{4t}-\alpha^2 t}}{4\pi  t}\,\mathrm{d}t.
\end{align}
We denote the integrand by
\begin{align}\label{H_ker}
    H(t,\bfx) = \frac{e^{-\frac{\norm{\bfx}^2}{4t}-\alpha^2 t}}{4\pi  t},
\end{align}
and insert this representation into the single layer potential \eqref{slp}. We then partition the time integral at a positive parameter $\delta$, so that
\begin{align}
    \slp[\sigma](\bfx) = \slp_L[\sigma](\bfx) + \slp_H[\sigma](\bfx),
\end{align}
where
\begin{align}\label{sl_sh}
    \slp_L[\sigma](\bfx) = \int_0^\delta\int_{\partial\Omega} H(t,\bfx-\bfx')\sigma(\bfx')\,\mathrm{d} s_{\bfx'} \mathrm{d}t, \quad \slp_H(\bfx) = \int_\delta^\infty\int_{\partial\Omega} H(t,\bfx-\bfx')\sigma(\bfx')\,\mathrm{d} s_{\bfx'} \mathrm{d}t
\end{align}
The first integral, $\slp_L$, represents the short-range, local contribution, while $\slp_H$ captures the smooth, long-range component. Following the terminology of \cite{fryklund2026lightweight}, we use the term \textit{history} to refer to the long-range contribution. This comes from the heat-kernel representation: small values of the integration variable $t$ correspond to contributions that are local in time, while large values of $t$ correspond to contributions from further back in history. 

In practice, the local part is further decomposed into an asymptotic part $\slp_{L,A}$ and $J$ dyadic refinement levels $\slp_{L,D}^{(j)}$, $j=1,\dots, J$. This is achieved by subdividing the local integral over $[0,\delta]$ as a sum of smaller time intervals, where each time interval is progressively refined. 

Let $\dhat = \delta/4^J$ for some integer $J\geq 0$. Then
\begin{align}\label{slp_dyad}
\begin{split}
    \slp_L[\sigma](\bfx) &= \slp_{L,A} +\sum_{j=1}^J \slp_{L,D}^{(j)}[\sigma](\bfx),
\end{split}
\end{align}
where
\begin{align}
    \slp_{L,A} &= \int_0^{\dhat}\int_{\partial\Omega} H(t,\bfx-\bfx')\sigma(\bfx')\,\mathrm{d} s_{\bfx'}\mathrm{d} t\label{sla}\\
    \slp_{L,D}^{(j)} &= \int_{\delta/4^j}^{\delta/4^{j-1}}\int_{\partial\Omega}H(t,\bfx-\bfx')\sigma(\bfx')\,\mathrm{d}s_{\bfx'}\mathrm{d}t.
\end{align}
To evaluate the contributions on each dyadic level, we use the identity
\begin{align}\label{equiv_split}
    \int_0^\delta \frac{e^{-\frac{\norm{\bfx}^2}{4t}-\alpha^2 t}}{4\pi  t}\,\mathrm{d}t = G^R(\bfx,\delta),
\end{align}
where \cite{fryklund2023integral}
\begin{align}\label{ewald_real}
    G^R(\mathbf{x},\delta) = \frac{1}{4\pi}K_0\left(\frac{\norm{\bfx}^2}{4\delta},\alpha^2\delta\right)
\end{align}
is the real-space part of an Ewald decomposition \cite{ewald1921}, and $K_n(\cdot,\cdot)$ denotes the incomplete modified Bessel functions of the second kind and $n$th order. 

Using \eqref{equiv_split}, the contribution from each dyadic interval can be expressed as a difference of two real-space kernels
    \begin{align}
        \int_{\delta/4^j}^{\delta/4^{j-1}}
          H(t,\bfx-\bfx')\, \mathrm{d}t
        =
        G^R(\mathbf{x}-\bfx', \tfrac{\delta}{4^{j-1}})
        -
        G^R(\mathbf{x}-\bfx', \tfrac{\delta}{4^{j}}).
    \end{align}
Therefore, at each level $j$, we define the difference kernel
\begin{align}\label{diffker_s}
    \mathcal{K}_j^{\slp}(r,\alpha,\delta) = \frac{1}{4\pi}\left(K_0\left(\frac{4^{j-2}r^2}{\delta},\alpha^2\frac{\delta}{4^{j-1}}\right)-K_0\left(\frac{4^{j-1}r^2}{\delta},\alpha^2\frac{\delta}{4^{j}}\right)\right),
\end{align}
so that
\begin{align}
        \slp_{L,D}^{(j)}[\sigma](\bfx) &= \int_{\partial\Omega} \mathcal{K}_j^{\slp}(\norm{\bfx-\bfx'},\alpha,\delta)\sigma(\bfx')\,\mathrm{d}s_{\bfx'}\label{sld}.
\end{align}
Each difference kernel is smooth and rapidly decaying, and hence no special quadrature techniques are needed for evaluation. 

\vspace{5mm}

\noindent\textbf{Double layer potential.} The double layer potential is decomposed analogously
\begin{align}
    \dlp[\mu](\bfx) = \dlp_L[\mu](\bfx)+\dlp_H[\mu](\bfx),
\end{align}
with
\begin{align}
    \dlp_L[\mu](\bfx) &= \int_0^\delta \!\int_{\partial\Omega}\frac{\partial}{\partial\nu_{\bfx'}}H_t(\bfx-\bfx')\mu(\bfx')\,\mathrm{d} s_{\bfx'}\mathrm{d} t, \quad \dlp_H[\mu](\bfx) = \int_\delta^\infty\!\int_{\partial\Omega}\frac{\partial}{\partial\boldsymbol{\nu}_{\bfx'}}H_t(\bfx-\bfx')\mu(\bfx')\,\mathrm{d} s_{\bfx'}\mathrm{d} t.
\end{align}
The local part $\dlp_L$ is subdivided in the same manner as in \eqref{slp_dyad} into an asymptotic part $\dlp_{L,A}$ and $J$ dyadic refinement levels $\dlp_{L,D}^{(j)}$, where
\begin{align}
    \dlp_{L,A} &= \int_0^{\dhat} \!\int_{\partial\Omega}\frac{\partial}{\partial\boldsymbol{\nu}_{\bfx'}}H_t(\bfx-\bfx')\mu(\bfx')\,\mathrm{d} s_{\bfx'}\mathrm{d} t,\label{dla}\\
    \dlp_{L,D}^{(j)} &= \int_{\partial\Omega}\boldsymbol{\nu}(\bfx')\cdot(\bfx-\bfx')\mathcal{K}_j^{\dlp}(\norm{\bfx-\bfx'},\alpha,\delta)\mu(\bfx')\,\mathrm{d}s_{\bfx'},\label{dld}
\end{align}
with the corresponding difference kernel, again smooth and rapidly decaying, 
\begin{align}\label{diffker_dlp}
    \mathcal{K}_j^{\dlp}(r,\alpha,\delta) =\frac{4^{j-2}}{2\pi\delta}K_1\left(\frac{4^{j-2}r^2}{\delta},\frac{\alpha^2\delta}{4^{j-1}}\right)-\frac{4^{j-1}}{2\pi\delta}K_1\left(\frac{4^{j-1}r^2}{\delta},\frac{\alpha^2\delta}{4^j}\right).
\end{align}

\noindent\textbf{Volume potential.} For the volume potential, we define 
\begin{align}
    \vp_L[f](\bfx) &=\int_0^\delta\int_{\Omega} H_t(\bfx-\bfx')f(\bfx')\,\mathrm{d} \bfx'\mathrm{d} t,\label{vl} \\
        \vp_H[f](\bfx) &=\int_\delta^\infty\int_{\Omega} H_t(\bfx-\bfx') f(\bfx')\,\mathrm{d} \bfx'\mathrm{d} t.
\end{align}
In this case, the local contribution is evaluated entirely using asymptotic expansions, i.e. $\vp_L=\vp_{L,A}$. The reasons for this are discussed at the end of Section \ref{sec:dyad}.

%%%%%%%%%%%%%%%%%%%%%%%%%%%%%%
%%%%%%%%%%%%%%%%%%%%%%%%%%%%%%
%%%%%%%%%%%%%%%%%%%%%%%%%%%%%%
\section{Discretization of the domain}\label{sec:disc}
To evaluate the layer and volume potentials numerically, we require discretizations and quadrature rules for both the domain $\Omega$ and its boundary $\partial\Omega$. The boundary quadrature is used for the single and double layer potentials, while the volume quadrature is used for the source term in the volume potential. As our kernel-split framework takes care of singularities and near singularities, we only need quadrature methods for the integration of smooth functions. The two discretizations are constructed independently of each other.

\subsection{Triangular cut–cell mesh and associated quadrature}\label{sec:tri_cut_mesh}

We embed the computational domain $\Omega$ in a uniform background triangular grid with mesh size $\Delta x$, where $\Delta x$ denotes the side length of a triangle in the grid. Each vertex of the background grid is classified as interior or exterior to $\Omega$ using the boundary representation. A triangle is retained as an interior element if all vertices lie inside $\Omega$, discarded if all vertices lie outside, and marked as a cut element if its vertices lie on both sides of $\partial\Omega$. Cut elements are, in general, polygonal and are therefore subdivided into triangles using a local triangulation procedure. The resulting collection of elements
\begin{align}
    \mathcal{T}_h = \{T_j\}_{j=1}^{M}
\end{align}
is a triangular cut–cell mesh that geometrically conforms to $\partial\Omega$, i.e.
\begin{align}
    \Omega \approx \bigcup_{j=1}^{M} T_j.
\end{align}
The volume potential is evaluated by quadrature over the elements of $\mathcal{T}_h$. On each triangle $T_j$, we use a Vioreanu-Rokhlin quadrature rule \cite{vioreanu2014spectra} on a reference triangle. In our numerical experiments, we use a rule with six quadrature nodes per triangle, which is sufficient for third-order accuracy. The reference nodes and weights are mapped to the physical element by an isoparametric mapping, and the physical quadrature weights include the corresponding Jacobian factor. For interior triangles, the mapping is affine, and the polynomial exactness of the reference quadrature rule is preserved. For boundary triangles, we use a quadratic isoparametric mapping in which the boundary curvature is represented by retaining the associated mid-edge shape function. Such mappings are standard in the finite element literature, see for example \cite{zienkiewicz2013finite}. The quadratic mapping gives a third-order accurate representation of the local boundary geometry in $\Delta x$. 

 This construction gives a set of $N$ quadrature nodes $\{\bfx'_j\}_{j=1}^{N}$ with corresponding weights $\{w'_j\}_{j=1}^{N}$, that approximates the integral of a smooth function $f(\bfx)$ over $\Omega$,
\begin{align}\label{vol_quad}
    \int_\Omega f(\bfx')\,\mathrm{d}\bfx' \approx \sum_{j=1}^N f(\bfx'_j) w_j'.
\end{align}

Cut–cell meshes are widely used in finite volume and finite element methods, but the presence of degenerate elements or high aspect ratios typically requires stabilization techniques or cell–merging strategies to maintain stability and good conditioning \cite{burman2025cut}. In contrast, the volume–potential formulation used here is largely insensitive to such geometric irregularities. This robustness comes from the fact that no differential operator is discretized on the mesh. The volumetric mesh is only used to provide quadrature nodes and weights for the source term in the volume potential, including the Fourier coefficients for the evaluation of the smooth long-range part. Since standard quadrature rules maintain their accuracy on distorted elements, degenerate or sliver triangles do not significantly affect accuracy or convergence in practice, and no stabilization procedures or cell-merging techniques are needed.

\subsection{Boundary discretization and quadrature}\label{sec:bound_disc}
The boundary $\partial\Omega$ is assumed to be a smooth closed curve, and the boundary discretization is independent of the volume mesh. The boundary is partitioned into $M_B$ panels (or chunks), 
\begin{align}
    \partial\Omega = \bigcup_{j=1}^{M_B} \Gamma_j,
\end{align}
and we assume each panel $\Gamma_j$ to have an analytical parameterization, $\gamma_j(t)$, scaled to $[-1,1]$. Each panel is equipped with an $n_B$-point Gauss-Legendre rule on $[-1,1]$.  
However, the method only requires a sufficiently smooth parametrization to permit evaluation of required derivatives, which may also be constructed from discrete data via interpolation. Then, each panel can be represented by a polynomial of degree $q_B\geq 2$, constructed from $q_B+1$ interpolation points. 

By the use of the composite Gauss-Legendre quadrature, through $\gamma_j(t_q)$, $j=1,\ldots,M_B$, $q=1,\ldots,n_B$,  we define $N_B=M_B n_B$ boundary nodes $\{\bfx_j\}_{j=1}^{N_B}$,
 with corresponding quadrature weights $\{w_j\}_{j=1}^{N_B}$, and 
 for a smooth function $g(\bfx)$ we approximate 
 \begin{align}
    \int_{\partial \Omega} g(\bfx')\,\mathrm{d} s_{\bfx'} \approx 
\sum_{j=1}^{N_B} g(\bfx_j) w_j.
\label{eq:bdry_quad}
\end{align}

%In our numerical experiments, we assume that the boundary $\partial\Omega$ is given by an analytic parametrization, which allows geometric quantities such as curvature to be evaluated directly. However, the method only requires a sufficiently smooth parametrization to permit taking the required derivatives, which may also be constructed from discrete data via interpolation. In both cases, the boundary discretization is independent of the volume mesh, meaning that distorted mesh elements do not affect the discretization of the boundary. 

%When an analytic parametrization is not available, the boundary $\partial\Omega$ is assumed to be a smooth closed curve and is partitioned into $M_B$ panels (or chunks)
%\begin{align*}
%    \partial\Omega = \bigcup_{j=1}^{M_B} \Gamma_j.
%\end{align*}
%Each panel is represented by a polynomial of degree $q_B\geq 2$, which ensures accurate evaluation of geometric quantities such as curvature. Each panel is equipped with $q_B+1$ Gauss-Legendre nodes $t_i$ on $[-1,1]$, and all panels are parameterized in terms of chordal arclength.

For the local asymptotic corrections, we will also need to locate the closest boundary point to a target point $\bfx$. To do so, we first identify the nearest discretization point on $\partial\Omega$ and use it to select a panel. The closest point on that panel is then obtained by solving
\begin{align*}
\frac{d}{dt}
\left\lVert
\mathbf{x}-\boldsymbol{\gamma}_j(t)
\right\rVert^2
=0, \qquad t\in [-1,1],
\end{align*}
This nonlinear equation is solved by Newton iteration. When an analytic parameterization is available, $\boldsymbol{\gamma}_j(t)$ and its derivatives are evaluated directly. Otherwise, $\boldsymbol{\gamma}_j(t)$ denotes the local polynomial approximation of the panel.

\subsection{Discretization of integral equation and evaluation of solution}
The boundary integral equations \eqref{bie_dir} and \eqref{bie_neu} are discretized using the Nyström method on the boundary quadrature nodes.
 As the double layer kernel has a bounded limit \eqref{dlp_diag_limit} and is smooth on the boundary $\partial \Omega$, we can discretize the double layer potential using the quadrature rule for smooth functions as defined in \eqref{eq:bdry_quad}. 

Let $\boldsymbol{\nu}_i$ and $\kappa_i$ denote the outward normal and curvature at the corresponding boundary node $\bfx_i$, $i=1,\ldots,N_B$.
We define the discrete double layer matrix $K \in \R^{N_b \times N_b}$ by
\begin{align}
    K_{ij} = w_j\frac{\alpha}{2\pi}K_1(\alpha\lVert\bfx_i-\bfx_j\rVert)\frac{(\bfx_i-\bfx_j)\cdot\boldsymbol{\nu}_j}{\lVert\bfx_i-\bfx_j\rVert}, \qquad i\ne j.
\end{align}
The diagonal entries are obtained from the limiting value \eqref{dlp_diag_limit},
\begin{align}
    K_{ii} = -\frac{w_i\kappa_i}{4\pi}.
\end{align}
Thus the discrete Dirichlet problem, collocated at the boundary nodes, is
\begin{align}
\label{dirichlet-discrete}
    \left(-\frac{1}{2}I+K\right)\boldsymbol{\mu} = \mathbf{g}-\mathbf{V},
\end{align}
where $\mathbf{g} \in \R^{N_B}$ is the Dirichlet data, 
and $\mathbf{V} \in \R^{N_B}$ the volume potential, both evaluated at the boundary nodes. 
For the Neumann problem, the discrete system is
\begin{align}
\label{neumann-discrete}
    \left(\frac{1}{2}I+K\right)\boldsymbol{u} = \mathbf{V}+\mathbf{S}_g,
\end{align}
where $\mathbf{S}_g\in \R^{N_B}$ denotes the single-layer potential with density given by the Neumann data $g$, evaluated at each boundary node. 

The resulting systems are solved using GMRES, to obtain $\boldsymbol{\mu}\in \R^{N_B}$ and $\boldsymbol{u}\in \R^{N_B}$, respectively. Once these quantities are known, we can evaluate the solution anywhere in the domain as a post-processing step, using the ansatz \eqref{sol_dirichlet} and \eqref{sol_neumann}, respectively. To evaluate the entries in $\mathbf{V}$ and $\mathbf{S}_g$, as well as the solutions, we will employ our kernel split scheme, that will be explained in the next two sections. Note that even though the double layer kernel has a bounded on-surface limit, it is still highly peaked for any target point close to the boundary, and will also be included in this framework. Hence, the specific limit \eqref{dlp_diag_limit} will not be used in the post-processing step.  

%%%%%%%%%%%%%%%%%%%%%%%%%%%%%%
%%%%%%%%%%%%%%%%%%%%%%%%%%%%%%
%%%%%%%%%%%%%%%%%%%%%%%%%%%%%%
\section{Evaluation of local part}\label{sec:local}

The local parts for the layer potentials are subdivided into an asymptotic part and dyadic refinement levels. For the volume potential, the local part is evaluated exclusively using local asymptotics.

\subsection{Asymptotic expansions}

\begin{figure}
\centering
    \begin{subfigure}[b]{0.3\textwidth}
    \centering
    \includegraphics[width=5.2cm]{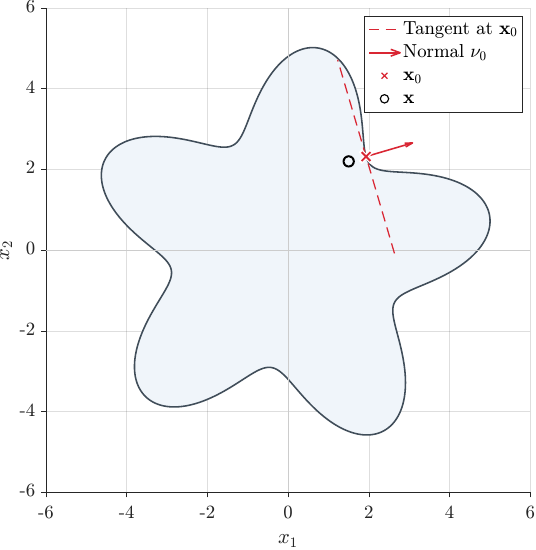}
    \caption{\label{fig:dom_unrot}}
    \end{subfigure}
\quad
    \begin{subfigure}[b]{0.3\textwidth}
    \centering
    \includegraphics[width=5.2cm]{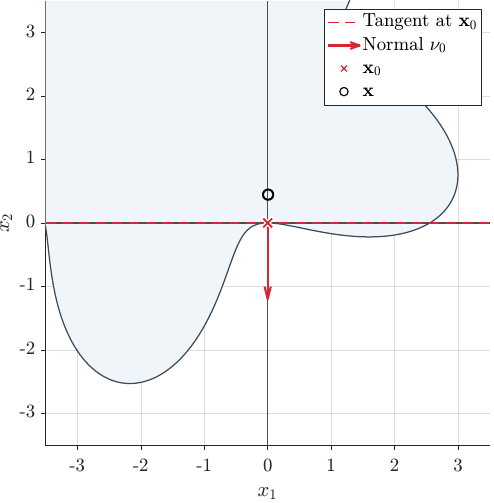}
    \caption{\label{fig:dom_rot_slp}}
    \end{subfigure}
\quad
    \begin{subfigure}[b]{0.3\textwidth}
    \centering
    \includegraphics[width=5.2cm]{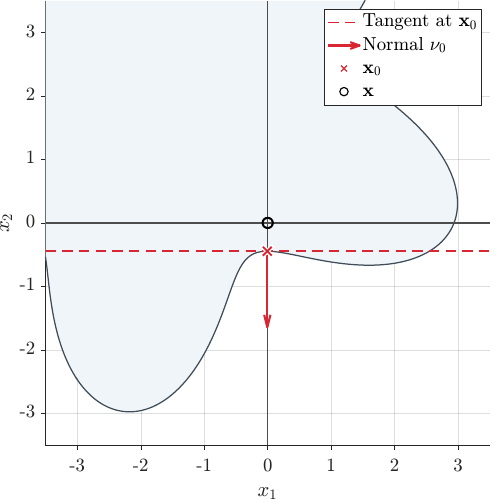}
    \caption{\label{fig:dom_rot_vp}}
    \end{subfigure}
\caption{\label{dom_rots} Illustration of the local coordinate transformation used in the asymptotic expansions: (a) original domain, (b) rotated and translated configuration for layer potentials, and (c) corresponding setup for the volume potential.}.
\end{figure}

We now derive the asymptotic formulas used for the local parts of the layer and volume potentials. The derivations have been carried out up to second order in $\delta_*$, but in practice we keep terms only up to first order, with a leading order error term of $\mathcal{O}(\delta_*^{3/2})$, to avoid the computation of high-order derivatives of the boundary geometry and layer densities.

We assume that $\delta_*$ is small enough so that every target point $\bfx$ within a distance of $\sqrt{\delta_*}(12+2\alpha\sqrt{\delta_*})$ from the boundary $\partial\Omega$ has a unique closest point $\bfx_0$ on the boundary. The distance $\sqrt{\delta_*}(12+2\alpha\sqrt{\delta_*})$ is chosen so that the Gaussian factor in \eqref{H_ker} is negligible, see also Remark \ref{rem:locality}.
To facilitate the asymptotic expansions, we introduce local coordinates obtained by translating and rotating the domain so that the tangent line at $\bfx_0$ is aligned with the $x_1$-direction. For the layer potentials, we place $\bfx_0$ at the origin, while for the volume potential we place $\bfx$ at the origin, see Figure \ref{dom_rots}. Denote by $\xi$ the excursion in the tangent line direction. By construction, we may approximate the boundary locally as $\partial\Omega \approx (\xi,\gamma(\xi))$. Since the kernel decays rapidly, only a neighborhood of the closest boundary point contributes to the local part. Therefore, we may restrict integration to a local neighborhood of $\bfx_0$ and extend it to $\mathbb{R}$, introducing only an exponentially small error. 

\vspace{0.5em}
\noindent\textbf{Single layer potential.}
We first consider the single-layer potential and derive its asymptotic expansion in the local coordinate system described above.

\begin{lem}\label{lem1_slp}
    Let $\bfx\in\mathbb{R}^2$ and assume there is a unique closest point $\bfx_0$ on $\partial\Omega$. Define $r = \norm{\bfx-\bfx_0}$ and $\rho = \frac{(\mathbf{x}-\mathbf{x}_0)\cdot\boldsymbol{\nu}_0}{\lVert \mathbf{x}-\mathbf{x}_0\rVert}$, where $\boldsymbol{\nu}_0$ is the outward normal at $\bfx_0$. Let $\kappa_0$ denote the curvature of the boundary at $\bfx_0$ and let $\sigma_0 = \sigma(\bfx_0)$. Introduce $c_1 = r/\sqrt{\dhat}$ and $c_2 = \alpha\sqrt{\dhat}$. Define
    \begin{align*}
        \Phi_- &= e^{-c_1c_2}\erfc\!\left(\tfrac{c_1}{2} - c_2\right)-e^{c_1c_2}\erfc\!\left(\tfrac{c_1}{2} + c_2\right),
        \\
        \Phi_+ &= e^{-c_1c_2}\erfc\!\left(\tfrac{c_1}{2} - c_2\right)+e^{c_1c_2}\erfc\!\left(\tfrac{c_1}{2} + c_2\right),
        \\
        Z &= \frac{4e^{-c_1^2/4-c_2^2}}{\sqrt{\pi}},
    \end{align*}
    where $\erfc$ denotes the complementary error function. 
    
    For $m=0,1,2,3$, define
    \begin{align*}
        P_m = c_1^m\frac{\Phi_-}{c_2},
    \end{align*}
    and, for $m=0,1$, define
    \begin{align*}
        Q_m = c_1^m\left[\frac{\Phi_--Z}{c_2^3}+\frac{c_1\Phi_+}{c_2^2} \right].
    \end{align*}
    Then \eqref{sla} can be written
    \begin{align}\label{slp_local1}
    \begin{split}
        \slp_{L,A}[\sigma](\bfx)=&\sqrt{\dhat}\frac{\sn}{4}P_0+\dhat\frac{\rho\kb\sn}{8}P_1+\dhat^{3/2}\left[\frac{\kb^2\sn}{32}(Q_0+3P_2)+\frac{\sn''}{8}Q_0\right]\\
        &+\dhat^2\rho\left[\frac{\kb^3\sn}{64}(-3Q_1+5P_3)+\left(\frac{3\kb\sn''}{16}+\frac{\sn'\yt}{8}+\frac{\sn\yf}{32}\right)Q_1\right]+\mathcal{O}(\dhat^{5/2}).
        \end{split}
    \end{align}
    Here, $\sigma_0'$ and $\sigma''_0$ denote the first and second derivatives of the density $\sigma$,
    %in the tangent line parameterization at $\bfx_0$, 
    and $\gamma_0^{(3)}$ and $\gamma_0^{(4)}$ denote the third and fourth derivatives of the curve $\partial\Omega$ at $\bfx_0$, all in the tangent line parameterization.
\end{lem}

\begin{proof}
Introducing the change of variables $z=\sqrt{4t}$ and $u=\xi/z$ in \eqref{sla}, we obtain
\begin{align}\label{sla_var}
\slp_{L,A}[\sigma](\bfx)
= \frac{1}{2\pi} \int_0^{2\sqrt{\dhat}} \int_{-\infty}^\infty
e^{-\alpha^2 z^2/4} e^{-u^2}
e^{-(\rho r - \gamma(uz))^2/z^2}
\sigma(uz)
\sqrt{1+(\gamma'(uz))^2}
\mathrm{d}u\mathrm{d}z.
\end{align}
The local boundary approximation and the layer density are expanded around the closest boundary point $\bfx_0$, which yields
\begin{align}
    \gamma(\xi) &\approx \frac{1}{2}\kappa_0 \xi^2 + \frac{1}{6}\gamma_0^{(3)}\xi^3+\frac{1}{24}\gamma_0^{(4)}\xi^4+\dots\label{x2_exp}\\
    \sigma(\xi) &\approx \sigma_0 + \sigma_0' \xi + \frac{1}{2}\sigma_0'' \xi^2 + \dots.
\end{align}
Inserting these expressions in \eqref{sla_var} and further expanding the exponential and arc length factor in powers of $z$, the integrand becomes a sum of terms of the form $u^k e^{-u^2}$ multiplied by polynomials in $z$. Odd terms vanish after integration over $u$, and the remaining terms can be evaluated explicitly. This reduces the expression to a one-dimensional integral in $z$, which can be computed in closed form in terms of exponentials and complementary error functions.

Using $z = \mathcal{O}\left(\sqrt{\delta_*}\right)$, and introducing the quantities in the statement of the lemma, finishes the proof.
\end{proof}

\begin{remark}\label{PmQm}
    The quantities
    \begin{align}
        P_m, &\qquad m=0,\dots,3 \label{Pm}\\
        Q_m, &\qquad m=0,1 \label{Qm}
    \end{align}
    are $\mathcal{O}(1)$ with respect to $\dhat$, assuming $\alpha$ is fixed. To see this, we Taylor expand the definitions of $P_m$ and $Q_m$ around $c_2 = 0$ to obtain
    \begin{align}
        P_m &=
        c_1^m\frac{\Phi_-}{c_2}
        =
        2c_1^m(G-c_1E)
        +
        \mathcal{O}(c_2^2),
        \qquad
        m=0,1,2,3, \label{Pmexp}\\
        Q_m & =
        \frac{2}{3}c_1^m
        \left[
            (2-c_1^2)G+c_1^3E
        \right]
        +
        \mathcal{O}(c_2^2),
        \qquad
        m=0,1, \label{Qmexp}
        \end{align}
    where 
    \begin{align}\label{defE}
        E(c_1) = \erfc\left(\frac{c_1}{2}\right),
        \qquad
        G(c_1) = \frac{2}{\sqrt{\pi}}e^{-c_1^2/4}.
    \end{align}
    Furthermore, we note that $E$ and $G$ have decayed to machine precision for $c_1 > 12$, so we can assume that $c_1 \in [0,12]$. We have that \(c_1^mE(c_1)\) and \(c_1^mG(c_1)\) are bounded on this interval, for $m=0,\dots,4$. Therefore the leading terms in the expansions \eqref{Pmexp} and \eqref{Qmexp} are bounded in $c_1$. Consequently,
    \begin{align}
        P_m=\mathcal{O}(1),
        \quad m=0,1,2,3, \qquad         Q_m=\mathcal{O}(1),
        \quad m=0,1,
    \end{align}
    with respect to $\dhat$.    
\end{remark}

\begin{remark}
\label{rem:locality}
    If the target point $\bfx$ satisfies $r\geq \sqrt{\delta_*}(12+2\alpha\sqrt{\delta_*})$, then $c_1/2-c_2\geq 6$. Consequently, the complementary error functions in the expressions for $\Phi_-$ and $\Phi_+$ are at machine precision level, as is $Z$. Therefore, the contribution from $\slp_{L,A}[\sigma](\bfx)$ is negligible for such target points, and we only need to evaluate $\slp_{L,A}$ at points within this distance from the boundary. 
\end{remark}

\noindent The following corollary is immediate.

\begin{cor}
    For target points $\bfx \in \partial \Omega$, the expression \eqref{slp_local1} simplifies to
    \begin{align}\label{slp_local2}
        \slp_{L,A}[\sigma](\bfx) & = \sqrt{\dhat}\frac{\sigma_0}{2c_2}\erf(c_2)+\dhat^{3/2}\frac{1}{c_2^3}\left(\frac{\kb^2\sn}{32}+\frac{\sn''}{8}\right)\left(2\erf(c_2)-\frac{4c_2}{\sqrt{\pi}}e^{-c_2^2}\right)+\mathcal{O}(\dhat^{5/2}).
    \end{align}
\end{cor}

\vspace{0.5em}
\noindent\textbf{Double layer potential.} The derivation for the double layer potential is analogous.

\begin{lem}
   Given the definitions as in Lemma \ref{lem1_slp}, furthermore define
    \begin{align*}
        U_m := c_1^m\Phi_+,
        \qquad
        m=0,\ldots,4.
    \end{align*}
    Then we can write \eqref{dla} as
    \begin{align}\label{dlp_local1}
    \begin{split}
    D_{L,A}[\mu](\bfx) =& -\frac{\rho\mu_0}{4}U_0-\sqrt{\dhat}\frac{\kb\mu_0}{8}(P_0+U_1)-\dhat \left(\frac{\rho\mu_0''}{8}P_1+\frac{3\rho\kb^2\mu_0}{32}(P_1+U_2)\right)\\
    &+\dhat^{3/2}\left[\left(\frac{5\kb^3\mu_0}{64}-\frac{\mu_0\yf}{32}\right)(3Q_0-U_3) +\frac{3\kb\mu_0''}{16}(-Q_0-P_2)+\frac{\mu_0'\yt}{8}(-2Q_0-P_2)\right]
    \\
    &+\dhat^2\left[\frac{35\rho\kb^4\mu_0}{512}(9Q_1+2P_3-U_4)-\frac{\rho\mu_0^{(4)}}{32}Q_1-\frac{15\rho\kb^2\mu_0''}{64}(Q_1+P_3)\right]\\
    &-\dhat^2\left[\frac{5\rho\kb\mu_0'\yt}{16}(2Q_1+P_3)+\left( \frac{5\rho\kb\mu_0\yf}{64} +\frac{5\rho\mu_0(\yt)^2}{96}\right)(3Q_1+P_3)\right] + \mathcal{O}(\dhat^{5/2}).
    \end{split}
    \end{align}
where $\mu_0 = \mu(\bfx_0)$, and $\mu'_0 = \mu'(\bfx_0)$, $\mu''_0 = \mu''(\bfx_0)$ and $\mu_0^{(4)} = \mu^{(4)}(\bfx_0)$ in tangent line parameterization.
\end{lem}

\begin{proof}
    The proof follows the same steps as for Lemma \ref{lem1_slp}, with the additional differentiation of the kernel.
\end{proof}

\begin{remark}
    $U_m$ is $\mathcal{O}(1)$ with respect to $\dhat$, for $m=0,\dots,4$. To see this, consider the Taylor expansion around $c_2 = 0$, 
    \begin{align*}
        U_m = 2c_1^mE+\mathcal{O}(c_2^2),
    \end{align*}
    with $E$ defined as in \eqref{defE}, where $c_1^m E$ is bounded, as discussed in Remark \ref{PmQm}.
\end{remark}

\begin{cor}
For any target point $\bfx \in \partial \Omega$, the expression \eqref{dlp_local1} simplifies to
\begin{align}\label{dlp_local2}
\begin{split}
    \mathcal{D}_{L,A}[\mu](\bfx) = &-\sqrt{\dhat}\frac{\kappa_0\mu_0}{4c_2}\erf(c_2)\\
    &+\dhat^{3/2}\frac{1}{c_2^3}\left(\frac{15\kb^3\mu_0}{64}-\frac{3\mu_0\yf}{32}-\frac{3\kb\mu_0''}{16}-\frac{\mu_0'\yt}{4}\right)\left(2\erf(c_2)-\frac{4c_2}{\sqrt{\pi}}e^{-c_2^2}\right)+\mathcal{O}(\dhat^{5/2}).
\end{split}
\end{align}
\end{cor}

\vspace{0.5em}
\noindent\textbf{Volume potential.} For the local part of the volume potential, we use the corresponding local coordinate system, but with $\bfx$ placed at the origin rather than the closest boundary point $\bfx_0$. Let $\xi$ denote the tangent line direction at $\bfx_0$, and let $\eta$ denote the direction of the inward normal at $\bfx_0$. Since we do not apply dyadic refinement to the local part of the volume potential, we use $\delta_* = \delta$.

\begin{lem}
Let $\bfx \in\Omega$ and assume there is a unique closest point $\bfx_0$ on the boundary $\partial\Omega$. Let $\kappa_0$ denote the curvature of the boundary at $\bfx_0$, $f = f(\bfx)$, $f^{(0,1)}=f_\eta(\bfx)$, $f^{(0,2)}=f_{\eta\eta}(\bfx)$ and $f^{(2,0)} = f_{\xi\xi}(\bfx)$. Use the same definitions as in Lemma \ref{lem1_slp}, and let
\begin{align*}
    \lambda &= e^{-c_2^2}\erfc\left(\frac{c_1}{2}\right),\\
    W_0 &= \frac{1}{c_2^2}\left(4-4e^{-c_2^2}+2\lambda-\Phi_+\right),\\
    W_1 &= \frac{1}{c_2^2}\left(2W_0-P_1+4(\lambda-2e^{-c_2^2})\right).
\end{align*}
Then \eqref{vl} can be written
\begin{align}\label{vp_local1}
\begin{split}
    V_L[f](\bfx) &= \delta_*\frac{f}{4}W_0+\delta_*^{3/2}\frac{2f^{(0,1)}-f\kappa_0}{8}Q_0\\
    &+\delta_*^2\left(\left(\frac{f^{(0,1)}\kappa_0}{8}-\frac{3f\kappa_0^2}{32}-\frac{f^{(0,2)}}{8}\right)Q_1+\frac{f^{(0,2)}+f^{(2,0)}}{8}W_1\right)+\mathcal{O}(\delta_*^{5/2}).
\end{split}
\end{align}
\end{lem}

\begin{proof}
    In the integral representation of the local part of the volume potential \eqref{vl}, we introduce the change of variables $z = \sqrt{4t}$, $u=x_1/z$ and $v = x_2/z$. Using the local approximation of the boundary, we can write \eqref{vl} as
    \begin{align}\label{vl_uvz}
        \vp_L[f](\bfx) &= \frac{1}{2\pi}\int_0^{2\sqrt{\delta_*}}\int_{-\infty}^\infty \int_{\frac{\gamma(uz)}{z}}^\infty ze^{-\frac{\alpha^2z^2}{4}}e^{-(u^2+v^2)}f(uz,vz)\mathrm{d}v\mathrm{d}u\mathrm{d}z.
    \end{align}
    A Taylor expansion of the source function $f$ around $\bfx$ is given by
    \begin{align}
        f(\bfx') = f+f_\xi\xi+f_\eta\eta+\frac{1}{2}f_{\xi\xi}\xi^2 + f_{\xi\eta}\xi\eta + \frac{1}{2}f_{\eta\eta}\eta^2 + \dots
    \end{align}
    We also Taylor expand $\gamma(uz)$ around $z=0$, and obtain
    \begin{align}
        \frac{\gamma(uz)}{z} = -\frac{\rho r}{z}+\frac{1}{2}\kappa_cu^2z+\frac{1}{6}\frac{\kappa'_c}{\lVert \gamma'_c\rVert}u^3z^2+\frac{1}{24}\gamma^{(4)}(0)u^4z^3+\frac{1}{120}\gamma^{(5)}(0)u^5z^4+\mathcal{O}(z^5),
    \end{align}
    since our system is such that $\gamma(0) = -r$ and $\gamma'(0) = 0$. 

    Similarly, we expand the functions
    \begin{align}
        e^{-(\gamma(uz)/z)^2},\qquad \erfc\left(\frac{\gamma(uz)}{z}\right),
    \end{align}
    which appear after performing the $v$-integration. 

    Inserting all expansions into the integral \eqref{vl_uvz} reduces the problem to Gaussian-weighted integrals in $u$, which can be evaluated analytically. This reduces $\vp_L[f](\bfx)$ to a one-dimensional integral in $z$, which can be computed analytically to render the final result. 
\end{proof}

\begin{remark}
    The quantities $W_m$, $m=0,1$ are $\mathcal{O}(1)$ with respect to $\delta$, since Taylor expansion around $c_2 = 0$ gives
    \begin{align}
         W_0 = 4-(2+c_1^2)E+c_1G+\mathcal{O}(c_2^2)
    \end{align}
    and
    \begin{align}
        W_1=4-2E+\frac{c_1G}{3}+\frac{c_1^4E-c_1^3G}{6}+\mathcal{O}(c_2^2),
    \end{align}
    with $E$ and $G$ defined as in \eqref{defE}. 
\end{remark}

\begin{cor}
    For target points $\bfx \in \partial \Omega$, \eqref{vp_local1} simplifies to
    \begin{align}\label{vp_local2}
    \begin{split}
        \mathcal{V}_L[f](\bfx) =& \frac{\delta_* f}{2c_2^2}(1-e^{-c_2^2})+\delta_*^{3/2}\frac{2f^{(0,1)}-f\kb}{8c_2^3}\left(2\erf(c_2)-\frac{4c_2}{\sqrt{\pi}}e^{-c_2^2}\right)\\
        &+ \delta_*^2\frac{f^{(0,2)}+f^{(2,0)}}{2c_2^4}(1-(1+c_2^2)e^{-c_2^2}) +\mathcal{O}(\delta_*^{5/2}).
    \end{split}
    \end{align}
\end{cor}

\begin{cor}
    For points at a distance $r\geq \sqrt{\delta_*}(12+2\alpha\sqrt\delta_*)$ from the boundary, some terms in \eqref{vp_local1} are negligible, leading to the expression
    \begin{align}\label{vp_local3}
        \mathcal{V}_L[f](\bfx) = \frac{\delta_* f_0}{c_2^2}\left(1-e^{-c_2^2}\right)+\delta_*^2\frac{f^{(0,2)}+f^{(2,0)}}{c_2^4}(1-(1+c_2^2)e^{-c_2^2})+\mathcal{O}(\delta_*^3).
    \end{align}
    In this case, the $\mathcal{O}(\delta_*^{3/2})$ and $\mathcal{O}(\delta_*^{5/2})$ terms are exponentially small, so the first neglected term is of order $\mathcal{O}(\delta_*^3)$.
\end{cor}

\begin{remark}
    For $\alpha = 0$, the equation \eqref{mod_helm} reduces to the Poisson equation. When taking the limit $c_2\to 0$ in equations \eqref{slp_local1}, \eqref{slp_local2}–\eqref{vp_local1} and \eqref{vp_local2}–\eqref{vp_local3}, they reduce to equations (38)-(44) in \cite{fryklund2026lightweight}.
\end{remark}

\subsection{Dyadic refinement}\label{sec:dyad}
The parameter $\delta$ controls the trade-off between accuracy and efficiency in the kernel decomposition. A smaller $\delta$ improves the accuracy of the local asymptotic expansion, but increases the number of Fourier modes required in the history part, making these computations costly. To overcome this, we introduce a dyadic refinement strategy based on a telescoping sum as in \cite{fryklund2026lightweight}. As described in \eqref{slp_dyad}, the local single layer potential is decomposed into an asymptotic contribution over the interval $[0,\delta_*]$ and a telescoping sum of dyadic correction terms over the intervals $[\delta/4^j, \delta/4^{j-1}]$. This approach enables the use of a small $\delta_*=\delta/4^J$ for the asymptotic part while keeping a larger, more efficient $\delta$ for the history part. 

%Each dyadic term in \eqref{sld} is evaluated using the difference kernel $\mathcal{K}_j^{\slp}$ defined in \eqref{diffker_s}. The double layer correction is treated analogously, using the difference kernel $\mathcal{K}_j^{\dlp}$ in \eqref{diffker_dlp}. These kernels have the same structure at each refinement level, but with a rescaled parameter. As $j$ increases, the kernel becomes more peaked, but the region where it must be evaluated shrinks by the same scaling. Therefore, the number of quadrature nodes needed to resolve the kernel remains the same.
Each dyadic term in \eqref{sld} is evaluated using the difference kernel $\mathcal{K}_j^{\slp}$ defined in \eqref{diffker_s}. The double layer correction is treated analogously, using the difference kernel $\mathcal{K}_j^{\dlp}$ in \eqref{diffker_dlp}. These kernels have the same structure at every refinement level, differing only by a rescaling of the argument. As $j$ increases, the kernel becomes more peaked, but its effective support shrinks proportionally, since the rapidly decaying tails may be truncated to meet a prescribed accuracy. Consequently, the number of quadrature nodes required to resolve each kernel remains constant.

The dyadic corrections are evaluated by locally refining the boundary panels close to the evaluation point through repeated bisection. At refinement level $j$, each relevant panel is split into two child panels. Each child panel is equipped with the same $n_B$-point Gauss-Legendre rule, mapped to the corresponding child interval. The geometry and density values at the child quadrature nodes are then obtained by interpolation from the polynomial representation on the parent panel. Since each bisection uses the same reference interpolation matrices, these matrices can be precomputed and reused at every refinement level. 

Since the difference kernels in \eqref{diffker_s} and \eqref{diffker_dlp} are rapidly decaying, only panels close to the target point are refined. The correction from the refined panels is then accumulated back onto the original boundary density values. Therefore the dyadic refinement does not introduce a new global boundary discretization, but is used only locally when evaluating the difference kernel corrections.

The difference kernels \eqref{diffker_s} and \eqref{diffker_dlp} contain the incomplete modified Bessel functions of the second kind, $K_n(\cdot,\cdot)$, that are difficult to evaluate efficiently to high accuracy. In this paper, we use the methods described in \cite{harris2009methods}. They yield at least ten digits of accuracy in the function evaluations, which is sufficient for our purposes. 

We find that, for the same $\delta$ parameter, the local contribution from the volume potential is typically significantly smaller than the corresponding contributions from the layer potentials, and the same is true for leading order error terms. Hence, it is possible to evaluate the volume potential with a larger $\delta$ than for the layer potentials, with a similar error contribution. Consequently, applying dyadic refinement to the volume potential is typically unnecessary, and we instead evaluate the local asymptotic expansion for the volume potential with $\delta$ instead of $\delta_*$. This is fortunate, as while the difference kernel for the volume potential is the same as for the single layer potential, evaluation at the dyadic refinement levels is more costly, since the correction becomes a two dimensional quadrature over the volumetric mesh rather than a one dimensional quadrature over boundary panels.

%In our numerical examples, we find that the local contribution from the volume potential is significantly smaller than the corresponding contributions from the layer potentials. Consequently, applying dyadic refinement to the volume potential is typically unnecessary, since it does not improve accuracy. We therefore evaluate the local part of the volume potential directly by the asymptotic expansion, with $\delta^* = \delta$, i.e. the same cutoff for the local part as for the history part. The difference kernel is the same for the volume potential as for the single layer potential, however, implementation of the dyadic refinement levels is more involved, since the correction becomes a two dimensional quadrature over the volumetric mesh rather than a one dimensional quadrature over boundary panels. This is possible in principle, but is not pursued in the present work. 

%%%%%%%%%%%%%%%%%%%%%%%%%%%%%%
%%%%%%%%%%%%%%%%%%%%%%%%%%%%%%
%%%%%%%%%%%%%%%%%%%%%%%%%%%%%%
\section{Evaluation of history part}\label{sec:hist}
The history part corresponds to the smooth, long-range contribution of the layer and volume potentials obtained from the split at $\delta$. Since the corresponding kernel is non-singular and rapidly decaying in Fourier space, it can be evaluated efficiently using the non-uniform fast Fourier transform (NUFFT).

\subsection{Fourier representation}
Taking the Fourier transform of the kernel \eqref{H_ker} with respect to the spatial variable gives
\begin{align}
    \widehat{H}(t,\mathbf{k}) = e^{-(\alpha^2+\norm{\mathbf{k}}^2)t}.
\end{align}
Using the inverse transform, we can write
\begin{align}
    H(t,\bfx-\bfx') = \frac{1}{(2\pi)^2}\int_{\mathbf{k}\in\mathbb{R}^2} e^{-(\alpha^2+\norm{\mathbf{k}}^2)t}e^{i\mathbf{k}\cdot (\bfx-\bfx')}\,\mathrm{d}\mathbf{k}.
\end{align}
Substituting this expression into the history part of the single layer potential in \eqref{sl_sh}, interchanging the order of integration, and integrating in time over $[\delta,\infty)$ gives
\begin{align}\label{sh}
    \slp_{H}[\sigma](\bfx) 
    &= \frac{1}{(2\pi)^2}\int_{\mathbf{k}\in\mathbb{R}^2}\frac{e^{-\delta(\alpha^2+\lVert \mathbf{k}\rVert^2)}}{\alpha^2+\lVert \mathbf{k}\rVert^2}\hat{\sigma}(\mathbf{k}) e^{i\mathbf{k}\cdot \bfx}\,\mathrm{d} \mathbf{k},
\end{align}
where
\begin{align}
\label{sigma_hat_def}
    \hat{\sigma}(\mathbf{k}) = \int_{\partial \Omega} \sigma(\bfx')e^{-i\mathbf{k}\cdot\bfx'}\,\mathrm{d} s_{\bfx'}
\end{align}
is the Fourier transform of the layer density. 

In practice, the Fourier integral \eqref{sh} is truncated to a finite square domain $[-k_{\max},k_{\max}]^2$. The cutoff $k_{\max}$ is chosen such that the exponential factor in the history multiplier is below a prescribed tolerance $\varepsilon>0$, namely $e^{-\delta(\alpha^2+k_\text{max}^2)}\leq\varepsilon$. Solving for $k_{\max}$ gives
\begin{align}\label{kmax}
    k_{\max} = \sqrt{\max\left(0,\frac{\ln(1/\varepsilon)-\delta\alpha^2}{\delta}\right)}.
\end{align}
If $\delta\alpha^2\geq\ln(1/\varepsilon)$, then the history contribution is already below the prescribed tolerance and we set $k_{\max} = 0$. 

Using the truncated history part yields
\begin{align}\label{slp_H}
    \slp_H [\sigma](\bfx) = \frac{1}{(2\pi)^2}\int_{[-k_\text{max},k_{\max}]^2}\frac{e^{-\delta(\alpha^2+\lVert \mathbf{k}\rVert^2)}}{\alpha^2+\lVert \mathbf{k}\rVert^2} e^{i\mathbf{k}\cdot \bfx}\hat{\sigma}(\mathbf{k})\,\mathrm{d} \mathbf{k} + \mathcal{O}(\varepsilon).
\end{align}
The history part of the volume potential has the same Fourier multiplier. The only difference is that the Fourier transform is taken over the volume source term, i.e.
\begin{align}
    \mathcal{V}_H[f](\mathbf{x})
    = \frac{1}{(2\pi)^2}
      \int_{[-k_\text{max},k_{\max}]^2}
      \frac{e^{-\delta(\alpha^2+\lVert\mathbf{k}\rVert^2)}}
           {\alpha^2+\lVert\mathbf{k}\rVert^2}
      e^{i\mathbf{k}\cdot\mathbf{x}}\,\hat{f}(\mathbf{k})\,\mathrm{d}\mathbf{k} + \mathcal{O}(\varepsilon),
\end{align}
where
\begin{align}\label{rhs_vol}
    \hat{f}(\mathbf{k}) =
    \int_{\Omega} f(\mathbf{x}') e^{-i\mathbf{k}\cdot\mathbf{x}'} \,\mathrm{d}\mathbf{x}'.
\end{align}
For the double layer potential, the normal derivative of the kernel introduces one additional factor of $\mathbf{k}$ in Fourier space. Therefore the cutoff is chosen slightly more conservative, so that 
\begin{align}\label{kmax_dlp}
    k_{\max}e^{-\delta(\alpha^2+k_{\max}^2)}\leq \varepsilon.
\end{align}
This gives
\begin{align}
    \mathcal{D}_H[\mu](\mathbf{x})
    = \frac{1}{(2\pi)^2}
      \int_{[-k_{\max},k_{\max}]^2}
      \frac{e^{-\delta(\alpha^2+\lVert\mathbf{k}\rVert^2)}}
           {\alpha^2+\lVert\mathbf{k}\rVert^2}
      e^{i\mathbf{k}\cdot\mathbf{x}}\,
      i\mathbf{k}\cdot\boldsymbol{\hat{\mu}}(\mathbf{k})\,\mathrm{d}\mathbf{k} + \mathcal{O}(\varepsilon),
\end{align}
where
\begin{align}
    \boldsymbol{\hat{\mu}}(\mathbf{k}) =
    \int_{\partial\Omega}\boldsymbol{\nu}(\mathbf{x}')\mu(\mathbf{x}')
    e^{-i\mathbf{k}\cdot\mathbf{x}'}\,\mathrm{d}s_{\mathbf{x}'}.
\end{align}

\subsection{Evaluation using NUFFT}
To evaluate the history contribution, we use the boundary and volume quadrature rules introduced in Sections \ref{sec:tri_cut_mesh} and \ref{sec:bound_disc}. The boundary quadrature rule \eqref{eq:bdry_quad} is used to compute discrete Fourier transforms of the layer densities, while the volume quadrature rule \eqref{vol_quad} is used to compute the discrete Fourier transform of the source function $f$.

\begin{comment}
To compute the discrete Fourier transforms of the layer densities, we need a quadrature rule over the boundary $\partial\Omega$, and to compute the discrete Fourier transform of the source function $f$, we need a quadrature rule over the entire domain $\Omega$. 
\textcolor{red}{Rewrite to use that discretization section has been moved up and can be refered to.}
Suppose therefore that we are given a set of $N_B$ boundary nodes $\{\bfx_j\}_{j=1}^{N_B}$ with corresponding quadrature weights $\{w_j\}_{j=1}^{N_B}$ that discretize the boundary $\partial\Omega$, and a set of $N$ nodes $\{\bfx'_j\}_{j=1}^{N}$ with corresponding quadrature weights $\{w'_j\}_{j=1}^{N}$, that discretize the computational domain $\Omega$. The role of the volumetric mesh is only to provide quadrature nodes and weights to approximate the integral \eqref{rhs_vol}. Therefore, poorly shaped or degenerate elements do not lead to stability issues. 
\end{comment}

We choose the Fourier truncation parameter $k_{\max}$ according to \eqref{kmax} and approximate the Fourier integrals over the square $[-k_{\max},k_{\max})^2$. The truncated Fourier domain is discretized by a uniform grid with $n_f$ points in each coordinate direction and spacing
\begin{align}\label{delta_k}
    \Delta k = \frac{2k_{\max}}{n_f}.
\end{align}
We denote the resulting Fourier nodes by $\{\mathbf{k}_l\}_{l=1}^{N_f}$, where $N_f = n_f^2$. 

We first consider the single layer potential. Discretizing the integral using \eqref{eq:bdry_quad}, 
the discrete Fourier transform of the layer density becomes
\begin{align}
\label{disrete_sigma_hat}
    \hat{\sigma}_l = \sum_{j=1}^{N_B} \sigma(\bfx_j)e^{-i\mathbf{k}_l\cdot\bfx_j}w_j.
\end{align}
The sum is efficiently evaluated using a type-1 NUFFT, which allows for non-uniform boundary nodes. The history part of the single layer potential is then approximated by 
\begin{align}\label{sh_fourier}
    \slp_H[\sigma](\bfx) \approx \frac{(\Delta k)^2}{(2\pi)^2}\sum_{l=1}^{N_f} \frac{e^{-\delta(\alpha^2+\norm{\mathbf{k}_l}^2)}}{\alpha^2+\norm{\mathbf{k}_l}^2}\hat{\sigma}_l \, e^{i\mathbf{k}_l\cdot\bfx}.
\end{align}
The inverse transform is evaluated at arbitrary target points using a type-2 NUFFT.

The double layer history contribution is evaluated analogously to the single layer history contribution. The normal derivative with respect to the source variable introduces an additional factor involving $\mathbf{k}$, and the Fourier truncation parameter $k_{\max}$ is chosen according to \eqref{kmax_dlp}, from which $\Delta k$ is chosen according to \eqref{delta_k}. The Fourier transform is then taken over the vector valued boundary density $\boldsymbol{\nu}(\bfx')\mu(\bfx')$. For the volume potential, the same Fourier grid and history multiplier as for the single layer potential are used, but the Fourier transform is taken over the whole domain $\Omega$, using the volume quadrature rule \eqref{vol_quad}, rather than over the boundary $\partial\Omega$. 

We need to determine what size of $\Delta k$ that is needed to resolve the inverse Fourier transform.  Suppose our domain $\Omega \subset [-L,L]^2$. This means that the largest possible separation between two points in $\Omega$ is $2L$. From the Nyquist condition we require of the Fourier grid that
\begin{align}
    \Delta k \leq \frac{\pi}{L}.
\end{align}
We are sampling $k$ on the interval
\begin{align}
    [-k_{\max},k_{\max}]
\end{align}
which has length $2k_{\max}$. Therefore, the number of grid points in one dimension must satisfy
\begin{align}\label{nf}
    n_f\geq \frac{2Lk_{\max}}{\pi} = \frac{2L}{\pi}\sqrt{\frac{\ln(1/\varepsilon)}{\delta}-\alpha^2}.
\end{align}
In addition, we need to have a fine enough spatial discretization to accurately evaluate the Fourier coefficients, i.e. to approximate the Fourier integral in the evaluation of $\hat{\sigma}_l$ in \eqref{disrete_sigma_hat} and correspondingly $\hat{f}_l$ for the source function.  
Consider the volumetric quadrature, where $\Delta x$ denotes the mesh size of the volumetric mesh. If we assume that we are uniformly sampling the function at points with a distance $\Delta x$, then the largest resolvable wavenumber given by the Nyquist frequency is
\begin{align}\label{k_nyq}
    k_{\text{Nyq}} = \frac{\pi}{\Delta x}.
\end{align}
Now, this is not exactly what we do. We do not have a uniform discretization, and as we have a third order quadrature rule, we are sampling the function at discrete points with spacing smaller than $\Delta x$. We do however find this measure useful, as we see that this is in practice close to the highest wave number that our meshes can actually resolve. 
\begin{comment}
    Consider the volumetric quadrature, where $\Delta x$ denotes the mesh size of the volumetric mesh, then the largest resolvable wavenumber is given by the Nyquist frequency
\begin{align}\label{k_nyqXXX}
    k_{\text{Nyq}} = \frac{\pi}{\Delta x}.
\end{align}
This condition comes from assuming that we are uniformly sampling the function at points with a distance $\Delta x$. This is not exactly what we do - 
\end{comment}

The computational complexity of evaluating the history parts for all three potentials is of the order $\mathcal{O}(N_f\log(N_f)+N_B+N)$ \cite{dutt1993fast}, where $N_f=n_f^2$.

%In addition, the volumetric quadrature must be fine enough to resolve the Fourier coefficients of the source function up to the truncation frequency. If $\Delta x$ denotes the mesh size of the volumetric mesh, then the largest resolvable wavenumber is given by the Nyquist frequency

\vspace{5mm}

\noindent \textbf{Influence of the parameter $\boldsymbol{\alpha}$.} As $\alpha$ increases, the kernel in \eqref{ker_greens} becomes increasingly localized and sharply peaked. This shifts the computational burden from the history part to the local part. As shown in \eqref{nf}, the number of required Fourier modes decreases as $\alpha$ grows; specifically, if $\alpha^2 \delta \geq \ln(1/\epsilon)$, the contribution from the history part becomes negligible. In this regime, we can set $k_{\max}=0$, effectively reducing the algorithm to the evaluation of the local part only.

Conversely, as $\alpha\to 0$, we have a singularity at the origin in Fourier space, making it difficult to evaluate the expression \eqref{sh_fourier}. This can be resolved, however, using the method of \cite{vico2016fast}. The limiting case $\alpha = 0$ corresponds to the Poisson equation and is treated in \cite{fryklund2026lightweight}. In this paper, we focus on the modified Helmholtz equation away from this limiting case. Empirically, we find that the number of Fourier modes \eqref{nf} is robust for $\alpha>2$, and all our numerical tests below use $\alpha\geq 10$. For use with elliptic marching, this is not a restrictive assumption as $\alpha^2$ is inversely proportional to the time step. 

%%%%%%%%%%%%%%%%%%%%%%%%%%%%%%
%%%%%%%%%%%%%%%%%%%%%%%%%%%%%%
%%%%%%%%%%%%%%%%%%%%%%%%%%%%%%
\section{Numerical results}\label{sec:num}
The algorithm described above is implemented using a combination of Fortran 77 and MATLAB. The computationally intensive parts, including the evaluation of the history part and the contributions at the dyadic refinement levels, as well as the identification of the closest boundary point and the interpolation of data to that point (needed in the asymptotic expansions), are written in Fortran, while the evaluation of the local part, the solver and the overall driver are handled in MATLAB. For the NUFFT calls in the history part, we use the FINUFFT library \cite{barnett2019parallel, finufft2018}.    

Except for the Hilbert tube example in Section \ref{sec:hilbert}, the computational domains are defined by the boundary curve
\begin{align}\label{bound_curve}
\gamma(t)=\left(r_0+ a\cos(\ell_1 t + \beta_1)+ b\sin(\ell_2 t + \beta_2)\right)e^{i\left(t + 0.1\sin(\ell_3 t)\right)},
\qquad t\in[0,2\pi],
\end{align}
where $r_0$, $a$, $b$, $\ell_1$, $\ell_2$, $\ell_3$, $\beta_1$, and $\beta_2$ are shape parameters. For each such domain, we generate a volumetric mesh using the method described in Section \ref{sec:tri_cut_mesh}. Volume integrals over the resulting triangles are approximated using six Vioreanu-Rokhlin quadrature nodes per triangle, which is sufficient for third-order accuracy. In all experiments, we set the tolerance for the NUFFT and the tolerance $\varepsilon$ for the history part to $10^{-6}$. 

In our first experiment, we validate the asymptotic expansions of the volume and layer potentials derived in Section \ref{sec:local}. The second example is a Neumann problem on a domain described by \eqref{bound_curve} with randomly chosen shape parameters, and the third example is a Dirichlet problem on a set of domains. This example demonstrates robustness with respect to mesh quality. Finally, the last example is a Dirichlet problem on the complicated Hilbert tube geometry.

\subsection{Validation of asymptotic expansions}

To validate the accuracy of the local asymptotic expansions for the layer potentials derived in Section \ref{sec:local}, we consider the homogeneous interior modified Helmholtz equation with boundary data generated from an exact solution. The computational domain is defined by \eqref{bound_curve} with parameters $r_0 = 1$, $a = 0.3$, $\ell_1 = 5$ and $b=\ell_3 = \beta_1 = 0$. 
%%This domain is not shown in a figure. 

We use a manufactured solution generated by a sum of Green's functions centered outside of the domain. 
Let $\bfx_j$, $j=1,\ldots,M$, denote locations outside of $\Omega$, and  define the exact solution as
\begin{align}\label{sol_hom}
    u(\bfx) = \sum_{j=1}^M c_jK_0(\alpha\lVert\bfx-\bfx_j\rVert), \quad \bfx\in\Omega, 
\end{align}
where $c_j$ are given coefficients. 
Since all Green's functions are centered outside of the domain, $u(\bfx)$ is smooth and satisfies the homogeneous modified Helmholtz equation in $\Omega$.

We first prescribe Dirichlet boundary data $g(\bfx)=u(\bfx)$, using \eqref{sol_hom} for $\bfx\in\partial\Omega$, 
%\begin{align}
%     g(\bfx) = \sum_{j=1}^M c_j  K_0(\alpha\lVert \bfx-\bfx_j\rVert), \quad %\bfx\in\partial\Omega.
%\end{align}
and compute the layer density $\mu$ by solving \eqref{bie_dir}, as discretized in 
\eqref{dirichlet-discrete}. 
Since $f=0$, the volume potential vanishes, and this problem can be solved using only the smooth quadrature. In the post-processing step, we use the kernel split to evaluate the double layer potential, using a discretization that makes sure that the local asymptotic expansion dominates the error (see discussion below). We evaluate the numerical solution at 10000 uniformly distributed target points inside the computational domain and measure the error compared to the exact solution $u(\bfx)$ in \eqref{sol_hom}.

\begin{comment}
To validate the asymptotic expansions for the double layer potential, we prescribe the Dirichlet boundary data
\begin{align}
     g(\bfx) = \sum_{j=1}^M c_j  K_0(\alpha\lVert \bfx-\bfx_j\rVert), \quad \bfx\in\partial\Omega.
\end{align}
The solution is then computed from the Dirichlet formulation \eqref{bie_dir}. Since $f=0$, the volume potential vanishes, and only the double layer potential needs to be evaluated. 
\end{comment}
To validate the asymptotic expansion for the single layer potential, we use Green's representation formula \eqref{sol_neumann}. Let
\begin{align}
\label{neu_data_hom}
    g(\bfx) = \frac{\partial u(\bfx)}{\partial \boldsymbol{\nu}} = \sum_{j=1}^M c_j \frac{\partial}{\partial \boldsymbol{\nu}} K_0(\alpha\lVert \bfx-\bfx_j\rVert), \quad \bfx\in\partial\Omega.
\end{align}
be the Neumann boundary data. Since $f=0$, Green's representation formula reduces to a sum of a single layer and a double layer potential. Hence the single layer potential satisfies,
\begin{align}
    \slp[g](\bfx) = u(\bfx)+\dlp[u|_{\partial\Omega}](\bfx),
\end{align}
where $u|_{\partial\Omega}$ is the boundary trace of $u$. We therefore construct reference values for the single layer potential $\slp_{\text{ref}}[g](\bfx)$, for all $\bfx$ points in the uniform grid interior to the domain, by evaluating the double layer potential with density $u|_{\partial\Omega}$ to high accuracy. For each value of $\delta$ in the convergence study, we then evaluate $\slp[g](\bfx)$ over the same set of $\bfx$ points and measure the error relative to the reference values. 

In our numerical tests, we use $M=10$ sources located at a distance of 0.2 outside the domain, and use uniform strengths $c_j = 10$. To ensure that the observed errors are dominated by the local asymptotic expansions, we use a highly resolved boundary discretization with 1000 panels and 16 Gauss-Legendre nodes per panel, and set all solver tolerances to $10^{-14}$, which ensures also that the eror in the evaluation of the long-range history part is negligible.  Numerical results for $\alpha=10$ are presented in Figures \ref{fig:local_conv_slp} and \ref{fig:local_conv_dlp}. We observe that the relative $l_\infty$ and $l_2$ errors follow the predicted theoretical convergence rates. 

To validate the asymptotic expansions for the volume potential $\mathcal{V}[f]$, we solve the inhomogeneous interior Dirichlet problem with source function 
\begin{align}
    f(\bfx)=\cos(2\pi x_1)\sin(2\pi x_2).
\end{align}
The Dirichlet problem is first solved using \eqref{bie_dir} to obtain the double layer density $\mu$. A high-accuracy reference solution $u_{\text{ref}}(\bfx)$ for the full boundary value problem is then computed using the method in \cite{fryklund2022adaptive}. Since the Dirichlet representation has the form \eqref{sol_dirichlet}, we isolate the volume potential contribution by subtracting an accurately evaluated double layer potential from the reference solution. More precisely, we define
\begin{align}
    \vp_{\text{ref}}[f](\bfx) = u_{\text{ref}}(\bfx)-\dlp_{\text{ref}}[\mu](\bfx),
\end{align}
where $\dlp_{\text{ref}}[\mu](\bfx)$ is evaluated with the present method tuned to high accuracy. The volume potential is then evaluated on a fixed, highly resolved volumetric mesh while varying the splitting parameter $\delta$. Errors are measured at the quadrature nodes of the volumetric mesh, by comparing with the reference volume potential. The use of a fixed fine mesh ensures that the observed error is dominated by the local asymptotic expansion rather than by the history part. As shown in Figure \ref{fig:local_conv_vp}, $\mathcal{V}_L[f]$ exhibits the expected order of convergence.
\begin{figure}
\centering
    \begin{subfigure}[b]{0.31\textwidth}
    \centering
    \includegraphics[width=5.6cm]{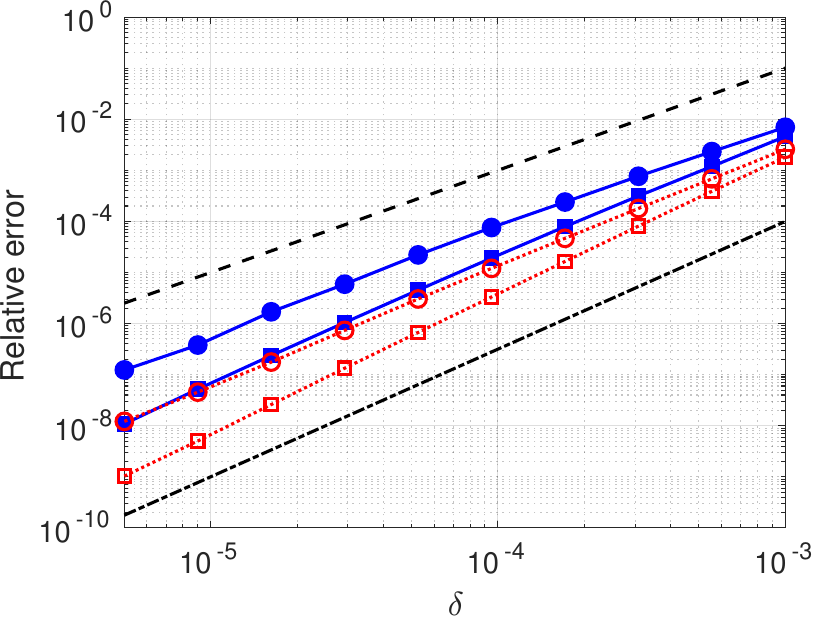}
    \caption{\label{fig:local_conv_slp}}
    \end{subfigure}
\quad
    \begin{subfigure}[b]{0.31\textwidth}
    \centering
    \includegraphics[width=5.4cm]{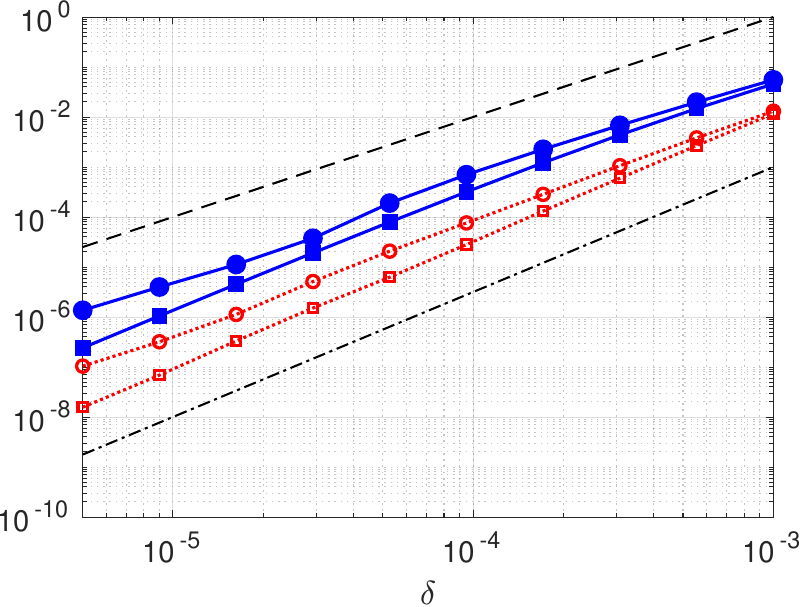}
    \caption{\label{fig:local_conv_dlp}}
    \end{subfigure}
\quad
    \begin{subfigure}[b]{0.31\textwidth}
    \centering
    \includegraphics[width=5.4cm]{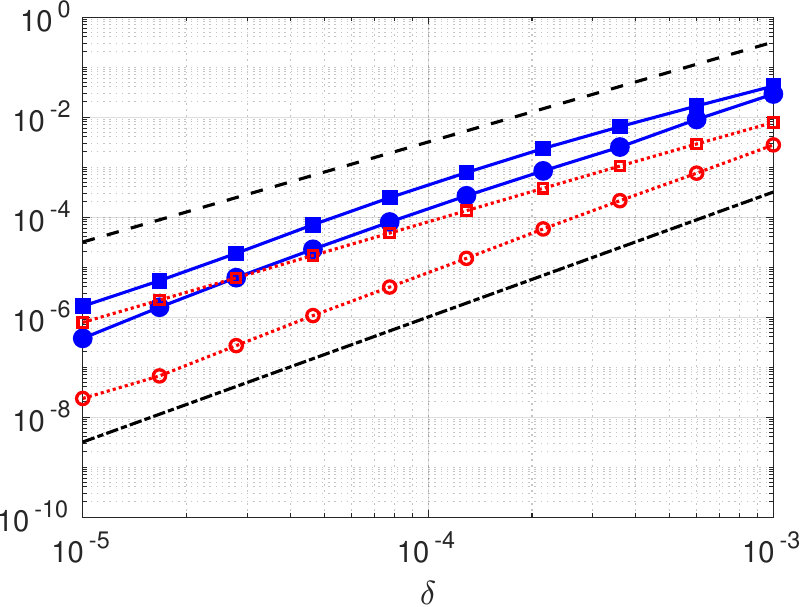}
    \caption{\label{fig:local_conv_vp}}
    \end{subfigure}
\caption{\label{local_conv} Relative errors as a function of the splitting parameter $\delta$ for the local part of the single layer potential $\slp[\sigma]$ (a), double layer potential $\dlp[\mu]$ (b), and volume potential $\vp[f]$ (c). Solid lines with filled markers denote relative $\ell_\infty$ errors, while dotted
lines with open markers denote relative $\ell_2$ errors. Square markers correspond to asymptotic expansions of order $\mathcal{O}(\delta^{2})$, while round
markers correspond to expansions of order $\mathcal{O}(\delta^{5/2})$. The dashed and dashed-dotted lines are reference slopes for  $\delta^{2}$ and $\delta^{5/2}$ convergence rates, respectively.}
\end{figure}

\vspace{0.5em}
\noindent\textbf{Influence of the Parameter $\boldsymbol{\alpha}$.} Figure \ref{slp_conv_alphas} illustrates the dependence of the local error for the single layer potential to the parameter $\alpha$. The increase in error with $\alpha$ is not explained by the dependence on $c_2=\alpha\sqrt{\delta_*}$ in the leading order $O(\delta_*^{3/2})$ error in the local asymptotic formula \eqref{slp_local1}. The terms that contain $c_2$ are $O(1)$, and the geometry is the same for all $\alpha$. 
Hence, the larger errors must come from the magnitude of the layer density $\sigma$, and its first and second derivatives, evaluated at the closest boundary point. 
This is clearly the case in this example. 
Here, the layer density is the Neumann data in \eqref{neu_data_hom}, shown in Figure \ref{fig:densities}. The oscillations are due to the $M=10$ terms in the sum, each one the normal derivative of a shifted modified Helmholtz kernel $K_0(\alpha r)$, that becomes more localized and more sharply peaked with increasing $\alpha$.
\begin{comment}
 Rather, as $\alpha$ increases, the modified Helmholtz kernel $K_0(\alpha r)$ becomes more localized and more sharply peaked. The corresponding layer density $\sigma$ is also $\alpha$-dependent, since it solves a boundary integral equation whose kernel depends on $\alpha$. As shown in Figure \ref{fig:densities}, larger values of $\alpha$ produce more sharply peaked derivatives. Since the neglected terms in the local expansion involve derivatives of the densities, these sharper densities lead to larger error constants, and hence to larger observed local errors.
   
\end{comment}
\begin{figure}[htbp]
    \begin{subfigure}{0.45\textwidth}
            \centering
    \includegraphics[width=7.2cm]{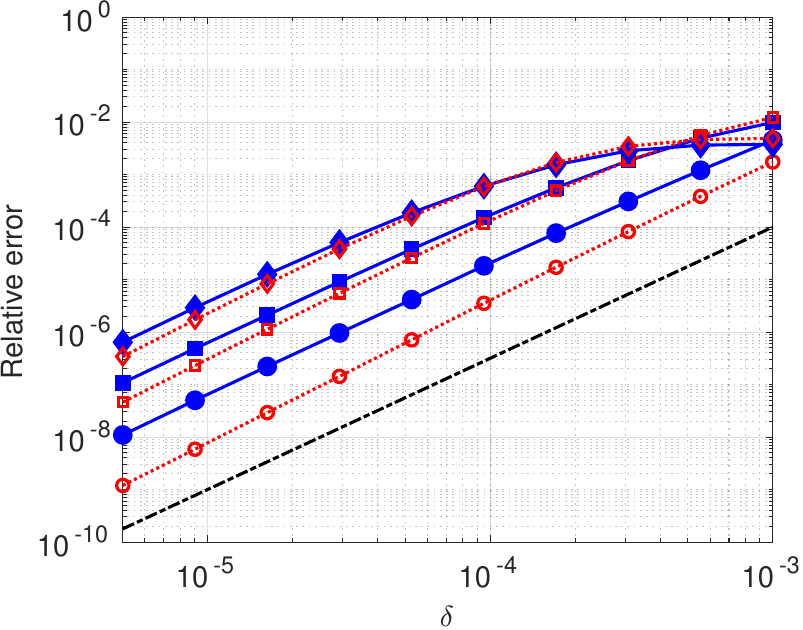}
    \vspace{1mm}
    \caption{\label{slp_conv_alphas} }
    \end{subfigure}
    \quad
    \begin{subfigure}{0.45\textwidth}
        \centering
        \includegraphics[width=7.2cm]{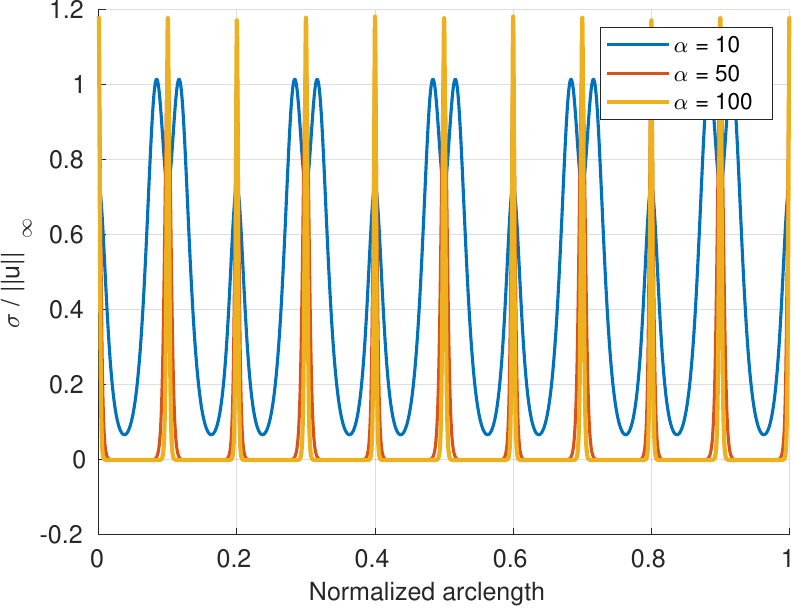}
        \caption{\label{fig:densities}}
    \end{subfigure}
\caption{(a) Relative $\ell_2$ (dotted) and $\ell_\infty$ (solid) errors as a function of the splitting parameter $\delta$ for the local part of the single layer potential $\slp_L[\sigma]$. Results are shown for $\alpha = 10$ (circles), $\alpha = 50$ (squares) and $\alpha = 100$ (diamonds). The dashed reference line corresponds to $\mathcal{O}(\delta^{5/2})$ convergence. (b) Corresponding single layer densities $\sigma$ on the boundary. As $\alpha$ increases, the modified Helmholtz kernel becomes more localized, producing more sharply peaked boundary densities. These larger density derivatives increase the constants in the local asymptotic error, explaining the higher error levels observed in (a).  }
\end{figure}

\subsection{Neumann problem on random smooth domain}\label{sec:int_neu}

The computational domain $\Omega$ is defined by the boundary curve \eqref{bound_curve} where the parameters $r_0 \approx 1.066 $, $a \approx 0.25$, $b$, $\ell_1 = 3$, $\ell_2 = 7$, $\ell_3 = 3$, $\beta_1\approx 3.89$, and $\beta_2\approx 1.61$ were chosen randomly. A triangular mesh for the domain with $\Delta x \approx 0.1$ is shown in Figure \ref{fig:neu_dom_mesh}.

\begin{figure}[h!]
\centering
    \begin{subfigure}[b]{0.45\textwidth}
    \centering
    \includegraphics[width=7.2cm]{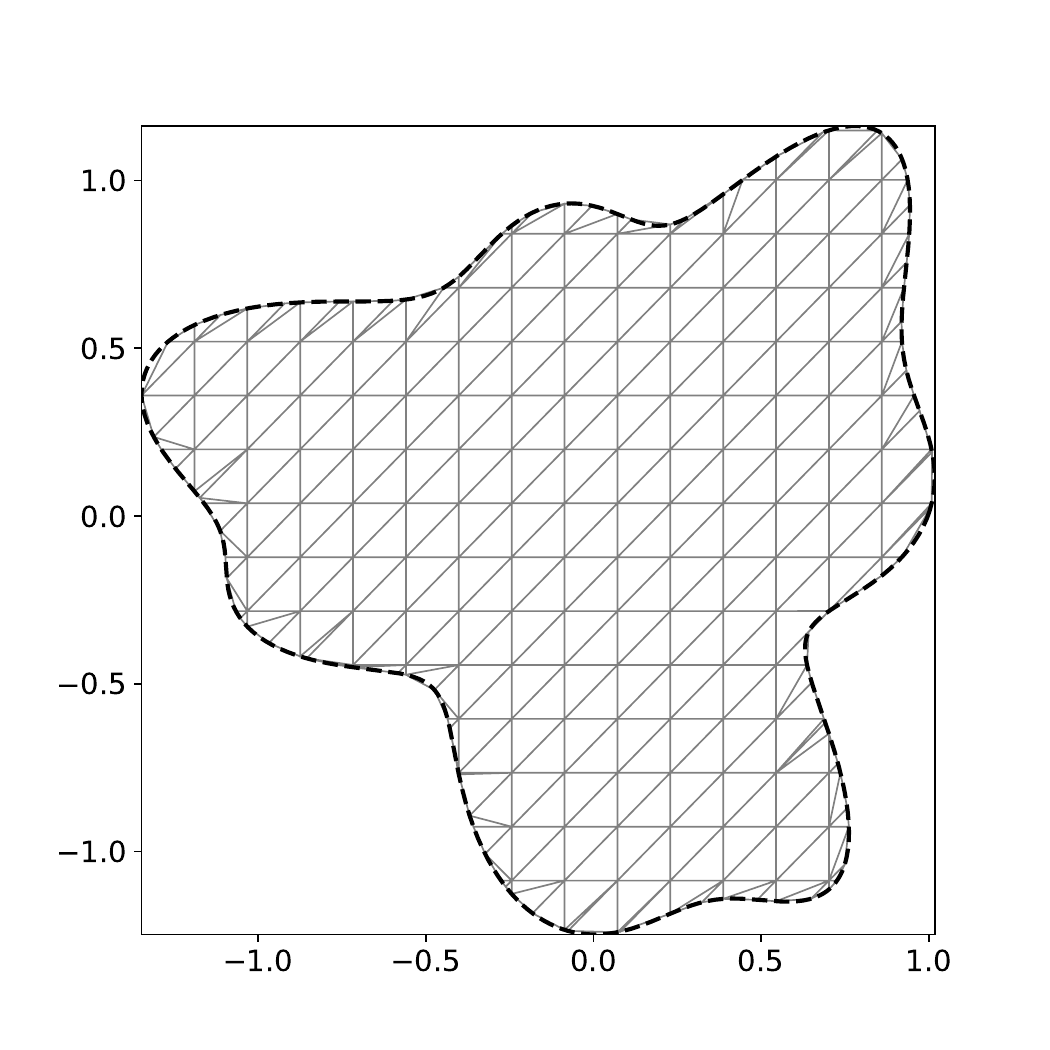}
    \caption{\label{fig:neu_dom_mesh}}
    \end{subfigure}
\quad
    \begin{subfigure}[b]{0.45\textwidth}
    \centering
    \vspace{1.5em}
    \includegraphics[width=7.5cm]{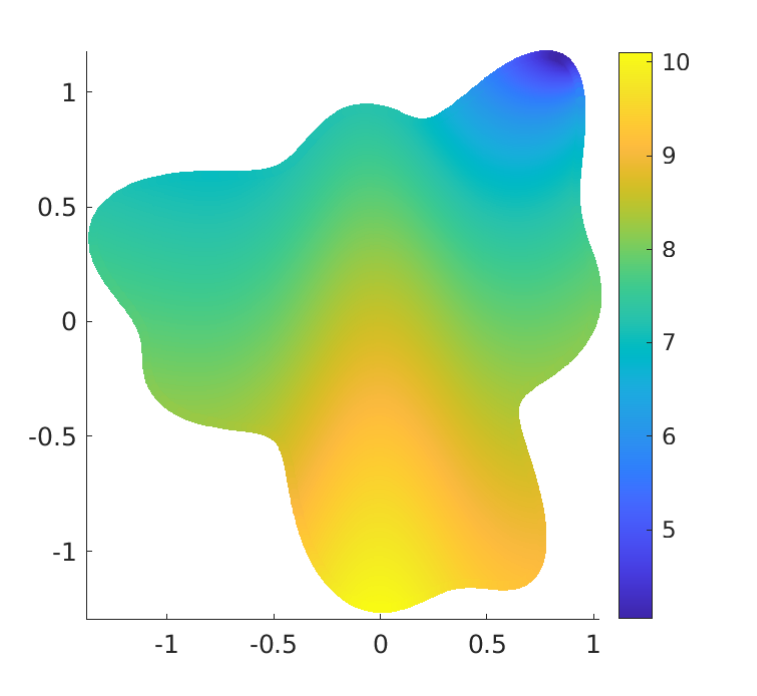}
    \caption{\label{fig:usol_int_neu}}
    \end{subfigure}
\caption{\label{dom_mesh} (a) Background triangular mesh cut to the domain $\Omega$, where polygonal boundary cells have been subdivided into triangles. Isoparametric mappings are used to conform with the curved boundary.
%retriangulated to conform to the boundary. 
The mesh size $\Delta x \approx 0.1$ corresponds to the side length of the interior regular triangles. (b) Manufactured solution \eqref{u_sol_int_neu} evaluated on the domain $\Omega$.}
\end{figure}

We consider the exact solution
\begin{align}\label{u_sol_int_neu}
    u(\bfx) = \frac{(x_1-p_1)}{(x_1-p_1)^2+(x_2-p_2)^2}+\frac{x_1^2}{9}-x_2+8+e^{-5x_1^2}, \quad \bfx\in\Omega
\end{align}
where $p_1=1.1$ and $p_2=1.3$. We analytically compute the source function $f(\bfx)$ and the Neumann boundary data $g(\bfx)$, so that \eqref{mod_helm} and \eqref{Neumann} are satisfied. Figure \ref{fig:usol_int_neu} shows the solution $u(\bfx)$ on the domain $\Omega$.

We solve the interior Neumann problem on a sequence of seven meshes with $\Delta x \in [1.6\times10^{-3}, \,2.4\times 10^{-2}]$. 
\begin{comment}
    The splitting parameter $\delta$ is chosen such that $\delta = 3\Delta x^2$, and the local expansions include terms up to $\mathcal{O}(\delta^{3/2})$, which means we get third order accuracy with respect to $\Delta x$ in the local part. To decrease the error in the local part, we use the dyadic refinement strategy as described in subsection \ref{sec:dyad}, with $J=3$ levels of refinement. 
\end{comment}
The splitting parameter $\delta$ is chosen such that $\delta = 3\Delta x^2$. We use the dyadic refinement strategy as described in subsection \ref{sec:dyad}, with $J=3$ levels of refinement, which means that $\delta_*$ is a factor of $4^3$ smaller than 
$\delta$, but still proportional to $\Delta x^2$. As the error in the local expansions are $\mathcal{O}(\delta_*^{3/2})$, we expect third order accuracy with respect to $\Delta x$.
The results for $\alpha = 10, 50$ and 100 are shown in Figure \ref{fig:int_neu1}, where we observe the expected convergence rate, until the tolerance for the history part is reached. 

Figure \ref{fig:int_neu2} shows the total time required for the post-processing step, in which the solution is evaluated at all nodes of the volumetric mesh. This step includes identifying the evaluation points that are close to the boundary, finding their closest boundary points, interpolating the geometric and density data to these points, and evaluating the local, history and dyadic contributions. For the four largest problem sizes shown in Figure \ref{fig:int_neu2}, the number of points processed per second per core is $1.1\times 10^6$, $1.2\times 10^6$, $1.3\times 10^6$ and $1.4\times 10^6$, respectively. 

\begin{figure}[h!]
\centering
    \begin{subfigure}[b]{0.45\textwidth}
    \centering
    \includegraphics[width=7.2cm]{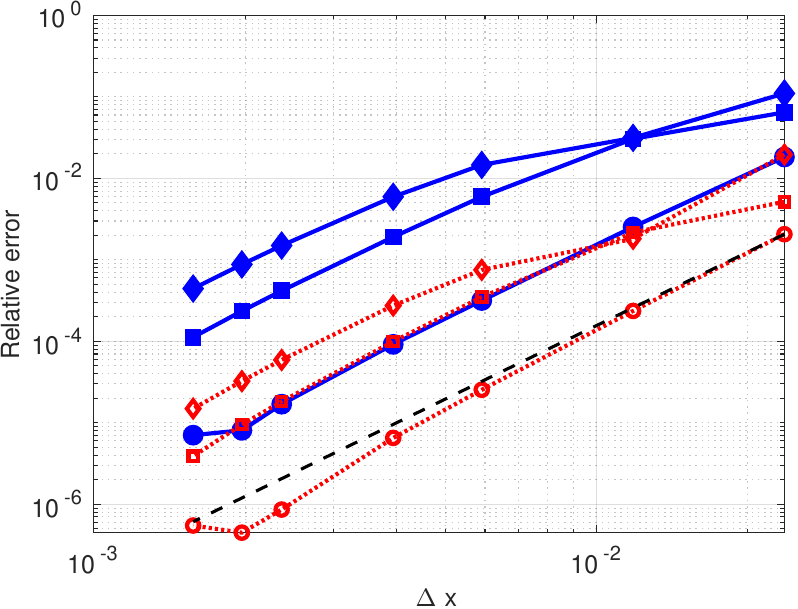}
    \caption{\label{fig:int_neu1}}
    \end{subfigure}
\quad
    \begin{subfigure}[b]{0.45\textwidth}
    \centering
    \includegraphics[width=7.2cm]{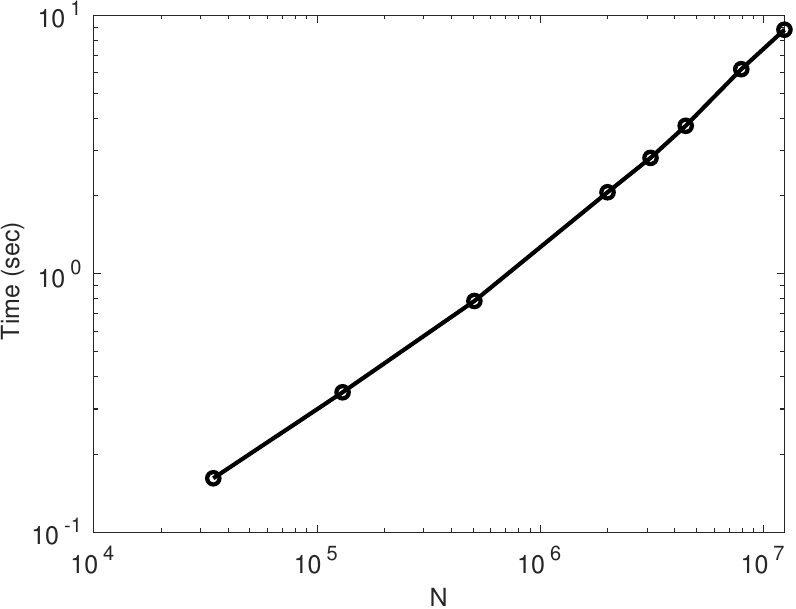}
    \caption{\label{fig:int_neu2}}
    \end{subfigure}
\caption{\label{int_neu} (a) Relative $\ell_2$ (dotted) and $\ell_\infty$ (solid) errors as functions of the mesh size $\Delta x$ for $\alpha = 10$ (circles), $\alpha = 50$ (squares) and $\alpha = 100$ (diamonds). The dashed line is a reference line corresponding to $\mathcal{O}(\Delta x^3)$. (b) Total post-processing time for evaluating the solution at all volumetric mesh nodes. This includes the evaluation of the single layer, double layer and volume potential contributions, including their local, history and dyadic parts, as well as the closest-point search and interpolation needed for near-boundary targets.}
\end{figure}

\subsection{Dirichlet problem with random shifts and rotations}\label{sec:dir_shift}

To assess the robustness of the method with respect to mesh quality, we consider the interior Dirichlet problem on a domain that is shifted and rotated multiple times. For each new shift, the domain boundary will intersect the underlying regular grid in a different way, possibly generating degenerate or sliver triangles.  
We investigate how rotation and translation of the domain affect the accuracy of the method, and thereby its sensitivity to highly distorted mesh elements.

The reference domain is defined by the boundary curve \eqref{bound_curve} with parameters $r_0 = 0.8$, $a=-0.05$, $b=0.08$,  $\ell_1=5$, $\ell_2 = 7$, $\ell_3 = 5$, $\beta_1=2$, and $\beta_2=5$. To generate a family of domains, we apply transformations of the form
\begin{align}\label{gamma_tilde}
    \widetilde{\gamma}(t) = \gamma(t)e^{i\phi} + \omega,
\end{align}
where $\phi \in [0,2\pi)$ is a rotation angle and $\omega \in \mathbb{C}$ is a translation. This means that the shape of the domain is preserved while its location in the plane changes. In Figure \ref{fig:total_layout} is plotted the domain for $\phi = 0.88$ and $\omega \approx 0.063+0.43i$, together with parts of the mesh.

\begin{figure}[htbp]
     \centering
     % --- LEFT COLUMN (Big Image) ---
     \begin{subfigure}[b]{0.48\textwidth}
         \centering
         \includegraphics[width=\textwidth]{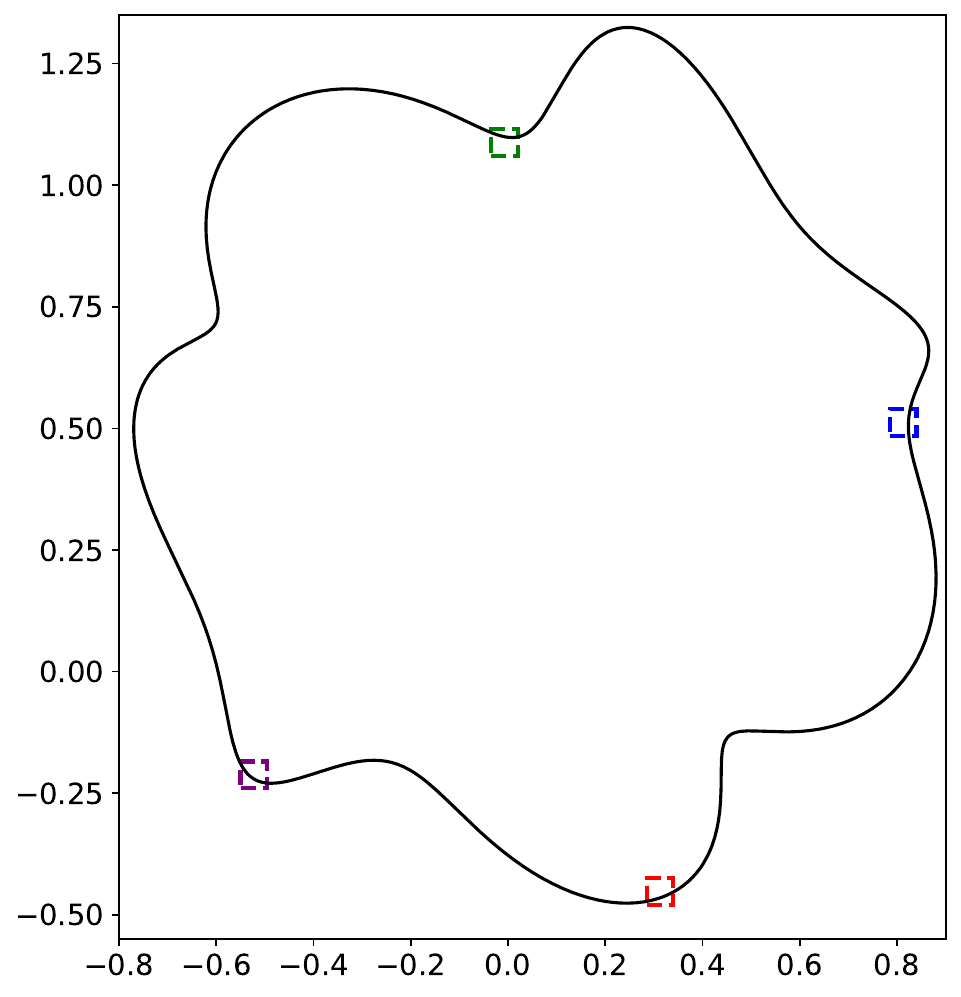}
         \label{fig:big}
     \end{subfigure}
     \hfill
     % --- RIGHT COLUMN (2x2 Grid) ---
     \begin{subfigure}[b]{0.48\textwidth}
         \centering
         % Top Row of Right Column
         \begin{subfigure}[b]{0.48\textwidth}
             \centering
             \includegraphics[width=\textwidth]{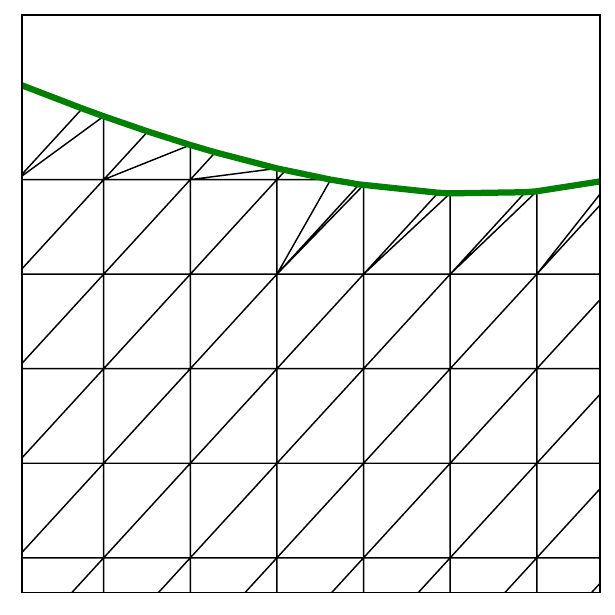}
         \end{subfigure}
         \hfill
         \begin{subfigure}[b]{0.48\textwidth}
             \centering
             \includegraphics[width=\textwidth]{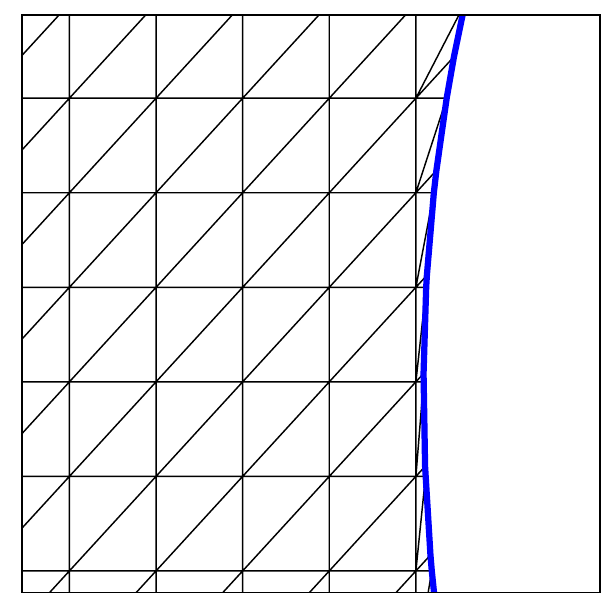}
         \end{subfigure}
         
         \vspace{10pt} % Space between the two rows
         
         % Bottom Row of Right Column
         \begin{subfigure}[b]{0.48\textwidth}
             \centering
             \includegraphics[width=\textwidth]{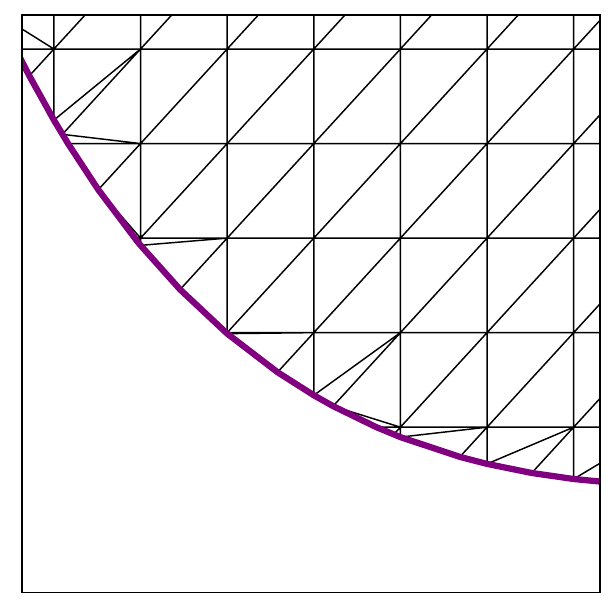}
         \end{subfigure}
         \vspace{10pt}
         %\hfill
         \begin{subfigure}[b]{0.48\textwidth}
             \centering
             \includegraphics[width=\textwidth]{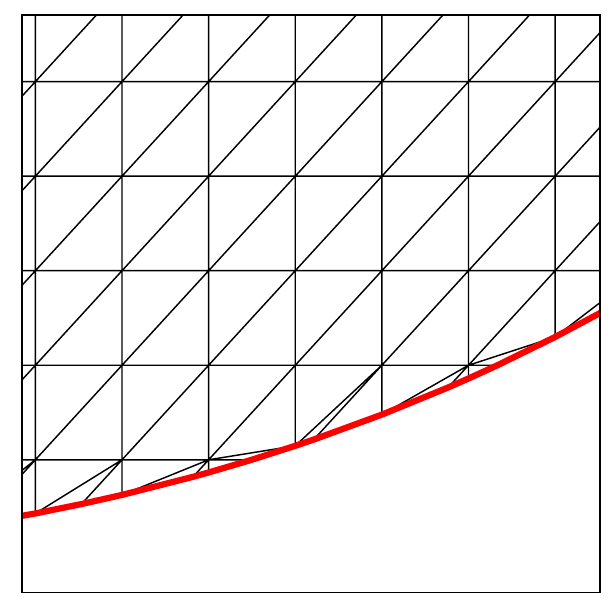}
         \end{subfigure}
         \vspace{10pt}
         \label{fig:small_grid}
     \end{subfigure}
     
     \caption{The domain $\Omega$ and selected regions of the triangular mesh near the boundary $\partial\Omega$. Since the boundary is allowed to intersect the background mesh arbitrarily, small and sliver elements may appear along the boundary.}
     \label{fig:total_layout}
\end{figure}

We consider the source function
\begin{align}
    f(\bfx) = -\sin(2(x_1+x_2)) + \frac{2}{9} + (10x_1^2 - 1)e^{-x_1^2},
\end{align}
and impose homogeneous Dirichlet boundary conditions $g(\bfx)=0$.

The problem is solved on a sequence of seven meshes with $\Delta x \in [1.6\times 10^{-3}, \, 8\times 10^{-3}]$ for 50 different values of $\phi$ and $\omega$. A reference solution is computed using the method of \cite{fryklund2022adaptive}. Figure \ref{fig:rndcurves} shows the relative $\ell_2$ and $\ell_\infty$ errors as functions of the mesh size for $\alpha = 10$. We observe stable convergence under domain rotation and translation, indicating that the method is largely insensitive to how the mesh intersects the underlying grid.

\begin{figure}[h!]
\centering
    \begin{subfigure}[b]{0.45\textwidth}
    \centering
    \includegraphics[width=7.2cm]{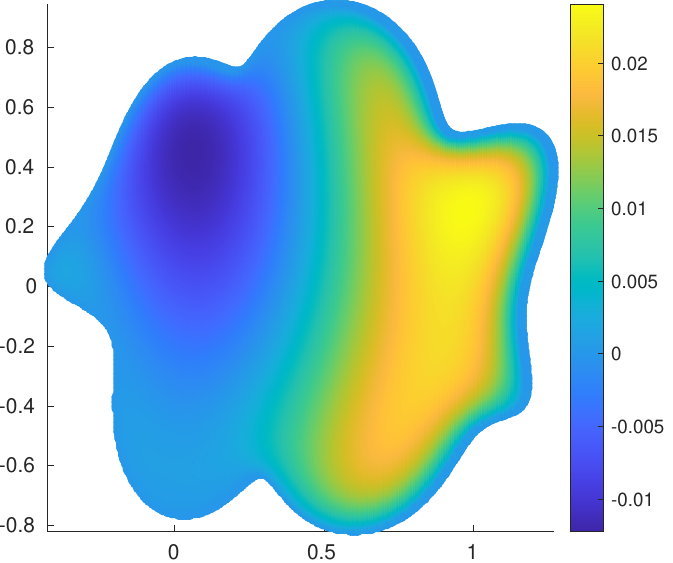}
    \caption{\label{fig:usol_rndcurves}}
    \end{subfigure}
\quad
    \begin{subfigure}[b]{0.45\textwidth}
    \centering
    \vspace{1.5em}
    \includegraphics[width=7.2cm]{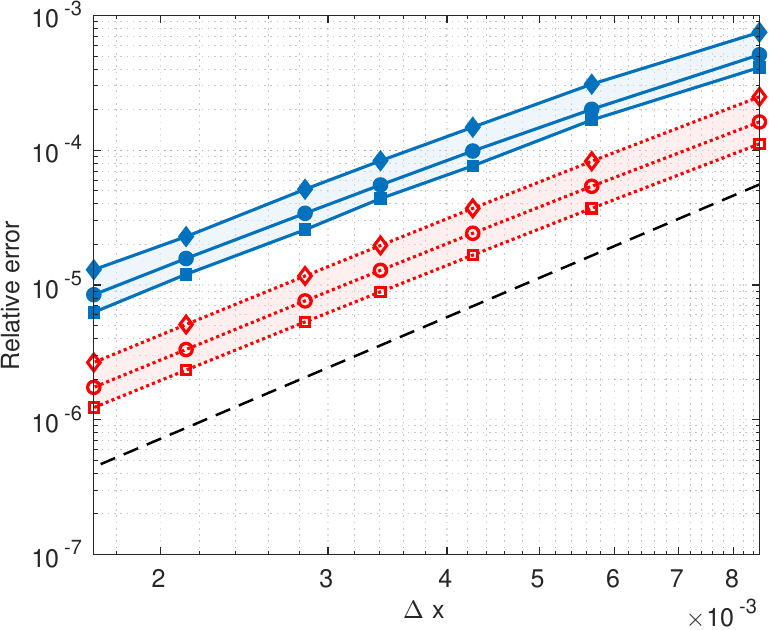}
    \caption{\label{fig:rndcurves}}
    \end{subfigure}
\caption{\label{rnd_mov} (a) Computed solution on one translated and rotated domain of the form \eqref{gamma_tilde}, with rotation angle $\phi = 0.04$ and translation $\omega \approx 0.45+0.038i$. (b) Relative $\ell_2$ (dotted red) and $\ell_\infty$ (solid blue) errors as functions of $\Delta x$, computed over 50 translated and rotated domains. For each value of $\Delta x$, square markers correspond to the smallest error over the 50 domains, diamond markers to the largest error and circular markers to the mean error. }
\end{figure}

To further assess robustness, we examine the quality of the cut–cell mesh by computing the aspect ratio
\begin{align}
    \mathrm{AR} = \frac{R}{2r},
\end{align}
where $R$ and $r$ denote the circumradius and inradius of each triangle, respectively. For an equilateral triangle, $\mathrm{AR}=1$, while large values indicate highly elongated elements.

For the rotated and shifted domain that yields the largest error, we compute the distribution of aspect ratios. On the coarsest mesh, the worst elements exhibit aspect ratios in the range $10^3-10^4$, corresponding to extremely degenerate sliver triangles. Despite the presence of these elements, the overall error remains stable. This confirms that the accuracy of the present method is largely insensitive to mesh quality.

\subsection{Dirichlet problem on Hilbert tube}\label{sec:hilbert}
Lastly, we consider a Hilbert tube domain. The construction is based on the finite approximating polygons of Hilbert's space-filling curve, as described in \cite{sagan1994space}. Rather than using the limiting space-filling curve, we use a finitely resolved Hilbert-type polygonal path. The recursive construction is stopped when the smallest feature width is 0.5. Following the construction used for Hilbert type geometries, the path is then mirrored to form a closed centerline. The resulting curve is translated so that its centroid is at the origin and scaled by a factor of 3.

To obtain a smooth periodic representation, we sample the closed centerline at 32768 points and retain 1024 Fourier modes. This gives a smooth periodic curve $\gamma_c(t)$, $t\in[0,2\pi]$. We define the corresponding unit normal by $ \boldsymbol{\nu}(t) = -i\frac{\gamma_c'(t)}{|\gamma_c'(t)|}$. The computational domain $\Omega$ is defined as a tubular neighborhood of $\gamma_c$ with total width 0.2. Equivalently, with half-width $H=0.1$, the inner and outer walls are 
\begin{align}\label{hilbert_dom}
    \partial\Omega_{\mathrm{out}} = \{\gamma_c(t)-H\boldsymbol{\nu}(t):t\in[0,2\pi]\}, \qquad \partial\Omega_{\mathrm{in}} = \{\gamma_c(t)+H\boldsymbol{\nu}(t):t\in[0,2\pi]\}. 
\end{align}
We consider the exact solution
\begin{align}\label{usol_hilbert}
    u(\bfx) = 1+0.2x-0.15y+\sin(2x+3y)+0.3\cos(5x-2y)+e^{-1.5(x^2+y^2)},
\end{align}
and set the source function $f(\bfx)$ and the Dirichlet boundary data $g(\bfx)$ so that \eqref{mod_helm} and \eqref{Dirichlet} are satisfied. The solution on the domain $\Omega$ is shown in Figure \ref{sol_hilbert}.

For our numerical experiment, we let $\alpha = 10$, $\Delta x \approx 2.5\times 10^{-3}$ and $\varepsilon = 10^{-6}$. From the Nyquist criterion \eqref{k_nyq}, we can determine the largest resolvable wavenumber $k_{\text{Nyq}}$. We choose $k_{\max} = 0.9\times k_{\text{Nyq}} \approx 1141$. From \eqref{kmax}, we can then determine the smallest history cutoff $\delta \approx 1.06\times 10^{-5}$ that can be used. With $J=3$ levels of dyadic refinement, we get $\delta_* \approx 1.66\times 10^{-7}$. Figure \ref{errplot_hilbert} shows the pointwise error distribution on the mesh. The $\ell_2$-error is approximately $1.34\times 10^{-5}$ and the $\ell_\infty$-error is approximately $6.61\times 10^{-4}$. We run the same test for $\alpha =  100$, with the same $\Delta x$. This gives us an $\ell_2$-error of approximately $3.19\times 10^{-5}$ and an $\ell_\infty$-error of approximately $6.92\times10^{-4}$.

\begin{figure}[h!]
    \begin{subfigure}{0.45\textwidth}
            \centering
    \includegraphics[width=6.0cm]{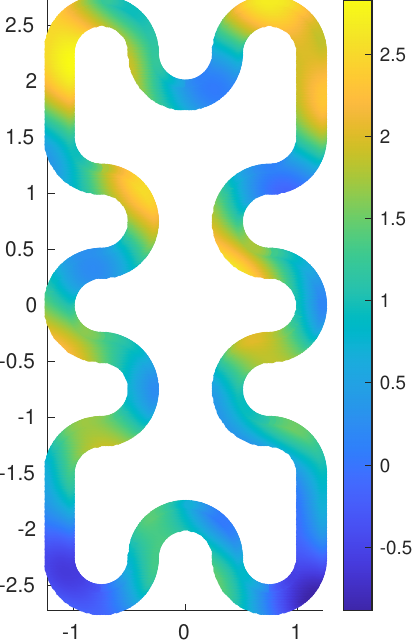}
    \vspace{1mm}
    \caption{\label{sol_hilbert} }
    \end{subfigure}
    \quad
    \begin{subfigure}{0.45\textwidth}
        \centering
        \includegraphics[width=5.8cm]{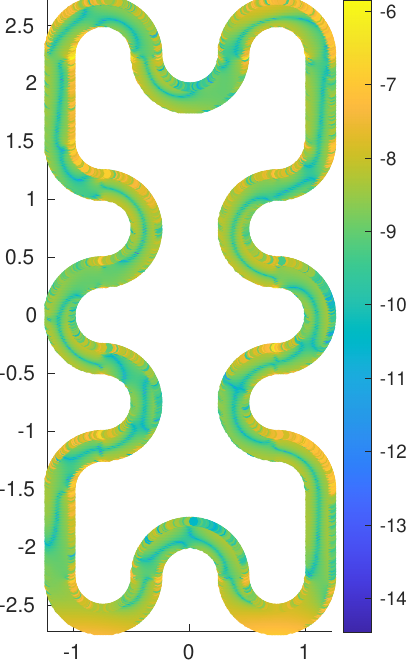}
        \caption{\label{errplot_hilbert}}
    \end{subfigure}
\caption{(a) Manufactured solution \eqref{usol_hilbert} evaluated on the Hilbert tube domain. The domain is bounded by the inner and outer walls defined in \eqref{hilbert_dom}. (b) Pointwise error distribution, shown as $\log_{10}$ of the magnitutde of the relative error, on the triangular cut-cell mesh. The error remains small throughout the narrow tube, also near the highly curved parts of the boundary.}
\end{figure}

%%%%%%%%%%%%%%%%%%%%%%%%%%%%%%
%%%%%%%%%%%%%%%%%%%%%%%%%%%%%%
%%%%%%%%%%%%%%%%%%%%%%%%%%%%%%

\section{Conclusions}

We have shown that layer and volume potentials for the modified Helmholtz equation can be treated within a single kernel-splitting framework that is both geometrically flexible and largely discretization agnostic. By using the same evaluation strategy across all target regimes and requiring only quadrature rules for integration of smooth functions over the boundary and the domain, the framework remains applicable even in complex or moving geometries.

The method extends the Poisson framework of \cite{fryklund2026lightweight} to the modified Helmholtz setting. It uses an integral representation of the free-space Green’s function, which splits the kernel into a short-range local part and a smooth long-range history part. The history contribution is evaluated in Fourier space using the NUFFT, while the local contribution is treated by asymptotic expansions. For layer potentials, the local part is further decomposed into dyadic refinement levels, leading to smooth and rapidly decaying difference kernels that can be resolved without specialized quadrature rules.

A key advantage of the framework is its compatibility with cut-cell meshes. Since the volumetric mesh is used only to provide quadrature nodes and weights for the source term, small or distorted cut cells do not introduce stiffness or conditioning issues. This makes the approach particularly attractive for elliptic marching methods for time-dependent problems, where implicit time discretizations, such as for the heat equation, lead to modified Helmholtz problems at each time step. In moving-domain settings, cells near the boundary can be updated on a fixed background grid, avoiding remeshing as the geometry evolves.

The numerical experiments demonstrate the expected convergence rates for both the local and history contributions, as well as efficient evaluation of layer and volume potentials. They also show robustness with respect to mesh geometry: translating or rotating a fixed domain relative to the background grid changes the shape and size of boundary triangles, yet has little effect on the observed errors. Even elements with extreme aspect ratios of $10^3-10^4$ only marginally affected the accuracy. The Hilbert tube example further demonstrates that the method remains accurate even for narrow, high-curvature domains with complicated boundaries.

The framework extends naturally to three dimensions, since the overall algorithmic structure remains unchanged. The main additional work lies in deriving and implementing the corresponding three-dimensional local expansions and dyadic refinement strategy. A promising direction for future work is to incorporate the present framework into a complete elliptic-marching solver for time-dependent problems, particularly in moving or geometrically complex domains.

\section*{Acknowledgement}
Edith Frisk Gärtner and Anna-Karin Tornberg gratefully acknowledge support from the Swedish Research Council under grant no 2023-04269. Fredrik Fryklund gratefully acknowledges support from the Knut and Alice Wallenberg Foundation under grant 2020.0258.

\newpage

\bibliographystyle{abbrv}
\bibliography{refs}

\end{document}